\documentclass{article}
\usepackage[authoryear]{natbib}
\usepackage{amsmath}
\usepackage{bm}
\usepackage{amssymb}
\usepackage{graphicx}
\usepackage{rotating}

\newtheorem{theorem}{Theorem}

\newtheorem{corollary}[theorem]{Corollary}

\newtheorem{lemma}[theorem]{Lemma}

\newtheorem{remark}[theorem]{Remark}

\newenvironment{proof}[1][Proof]{\noindent \textbf{#1.} }{\  \rule{0.5em}{0.5em}}
\input{tcilatex}

\begin{document}

\title{\textbf{Robust estimators for the log-logistic model based on ranked
set sampling }}
\author{Felipe, A.; Jaenada, M., Miranda, P. and Pardo, L. \\
$^{1}${\small Department of Statistics and O.R., Complutense University of
Madrid, Spain}}
\date{}
\maketitle

\begin{abstract}
In this paper we introduce a new family of estimators for the parameters of shape and scale of the log-logistic distribution being robust when rank set sample method is used to select the data. Rank set sampling arises as a way to reduce the impact of extremal data. Log-logistic distribution is an important distribution suitable for modeling many different situations ranging from Economy to Engineering and Hydrology. This new family of estimators is based on density power divergences. The choice of this family of divergence measures is motivated by the fact that they have shown a very good behavior in terms of robustness at a reduced cost in efficiency. This new family recovers the classical maximum likelihood estimator as a special case. We have developed the form of these estimators and derived their corresponding asymptotic distribution. A simulation study is carried out, suggesting that these new estimators are very robust when contamination arises and are competitive with classical estimators in terms of efficiency.
\end{abstract}

\noindent \underline{\textbf{Keywords}}\textbf{: }Log-logistic
distribution, rank set sample, Minimum density power divergence estimator, robustness, efficiency.

\noindent \underline{\textbf{AMS 2001 Subject Classification}}\textbf{: } 62F35, 62J12.

\section{Introduction}

In this paper we deal with the log-logistic distribution. This distribution arises as the logarithmic transformation of the logistic distribution and it is a special case of the Burr type-II distribution family (\cite{bur42}) and the Kappa distribution family (\cite{mijo73}). This distribution was first studied in (\cite{bai74}) and it has been shown that it is a suitable model for many different situations. For example, it has been applied in Hydrology (\cite{shmitr88}) as a good option to model flows in the rivers; it is a well-known distribution in Economics in the study of wealth or income distributions, where it is known as the Fisk distribution (\cite{fis61}). Other fields where log-logistic distribution has been applied ranges from Engineering (\cite{asma03}) to Cancer studies (\cite{guaklv99}). In general, log-logistic distribution deals with variables whose rate first increases and later decreases.

Log-logistic distribution depends on two parameters $\alpha , \beta .$ Parameter $\alpha $ models the scale of the distribution and coincides with the median, while parameter $\beta $ is a shape parameter that controls the form of the distribution. A problem in the practical use of log-logistic distribution is the estimation of parameters $\alpha , \beta .$ The classical method is maximum likelihood estimation (MLE) and estimators using MLE can be found in (\cite{bama87,kasr02,redowa18}).

Although MLE provides asymptotic efficient estimations, it is also common knowledge that MLE provides non-robust estimators. To deal with this problem, several other estimators have been proposed in the literature, mainly based on position measures that are more robust to contamination. We name here as example (\cite{asma03,guaklv99,shmitr88,abta16,hechya21}). A survey of these procedures can be found in (\cite{mawapa23}), where it is also introduced the RM-estimator.

In a recent paper (\cite{fejamipa24}), we have proposed the minimum density power divergence estimators (MDPDE) based on density power divergence (DPD) as a robust procedure to estimate the parameters of the log-logistic distribution. This family of divergences was first presented in (\cite{bahahjjo98}) and recovers as a special case the Kullback-Leibler divergence. It depends on a tuning parameter $\tau $ that models the trade-off between efficiency and robustness. Hence, small values of $\tau $ lead to very efficient but non-robust estimators, while large values of $\tau $ give way to robust although not very efficient estimators. MDPDE have proved to be efficient and very robust in many different situations, so that they seem to be a logical choice when searching robust estimators being at the same time efficient. In (\cite{fejamipa24}), we have shown that there exist values of $\tau $ that provide estimations that are competitive with MLE and are at the same time very robust in the presence of outliers.

In this paper, we deal with the case in which data have been obtained via ranking set sample (RSS) procedure instead of random sampling. This way to build samples was first introduced by McIntyre in (\cite{mci52}) and it is aimed to obtain samples avoiding extreme values. Hence, it is expected to obtain samples where sample (possible) contamination might be reduced. McIntyre observed that the relative efficiency, i.e. the ratio between the variance of the mean of a simple random sample of size $k$ and the variance of the mean of a ranked set sample of the same size, is not much less than (k+1)/2 for symmetric or moderately asymmetric distributions. He also noticed that the relative efficiency diminishes with increasing asymmetry of the underlying distribution but it is always greater than 1. A first study of the theoretical properties of RSS appears in (\cite{tawa68}). Since then, many papers applying RSS have been published. We mention here as a glimpse (\cite{decl72,sto80,kvsa93,chxiwu16,prbo99,abmu96,lich97,kapata97,alal03,chxiwu13}).

In (\cite{hechqi20}) MLE for the parameters of the log-logistic distribution based on RSS is presented. The main purpose of this paper is to introduce and to study the MDPDE for the parameters of the log-logistic distribution and to study the robustness of the corresponding new family of estimators when RSS is used.

The rest of the paper goes as follows: In Section 2 we introduce the MDPDE based on RSS for the parameters of the log-logistic model. The asymptotic distribution of the family of estimators introduced in the previous Section is studied in Section 3. In Section 4, the results of a simulation study to show the behavior of these new estimators and compare it with other estimators appearing in the literature are presented. We finish with the conclusions and open problems. The proofs of the theoretical results established in the paper are given in the Supplementary Material. In the Supplementary Material we also include more tables related to the simulation study.

\section{The minimum density power divergence estimator for the log-logistic
model}

\subsection{Basic concepts}

A random variable $X$ is said to have a {\bf log-logistic distribution} if the
probability density function is given by

\begin{equation}\label{densidad}
f_{\alpha ,\beta }(x)=\frac{\beta \alpha ^{\beta }x^{\beta -1}}{\left(
x^{\beta }+\alpha ^{\beta }\right) ^{2}},\text{ }x>0.
\end{equation}

The parameter $\alpha >0$ in this distribution is a scale parameter and is also the median of the
distribution. The parameter $\beta >0$ is a shape parameter. The
distribution is unimodal when $\beta >1$ and its dispersion decreases as
beta increases. The log-logisitc distribution has been applied in many different fields, especially in Economics.

Let us now introduce the {\bf rank set sample procedure} (RSS) introduced in (\cite{mci52}). Let $X_{1}^{1},...,X_{n}^{1}$ be a
random sample of size $n$ from the random variable $X$ and let us denote by $Y_{1}$ the first order statistic (or the smallest order statistic), i, e, $Y_{1}\equiv X_{(1)}^{1}=\min \left( X_{1}^{1},...,X_{n}^{1}\right) .$ Consider $X_{1}^{2},...,X_{n}^{2}$ a second random sample of size $n$ from the
random variable $X$ and denote by $Y_{2}\equiv X_{(2)}^{2}$ the second order
statistics. In general, we consider the $i$-th random sample from $X,$  $X_{1}^{i},...,X_{n}^{i}$ and we denote by $Y_{i}\equiv X_{(i)}^{i}$ the
$i$-th order statistic. We continue in the same way until to get a $n$-th random
sample from $X,$ $X_{1}^{n},...,X_{n}^{n}$ and we denote the largest order
statistic by $Y_{n}\equiv X_{(n)}^{n}.$ This way of selecting individuals in the final sample has been proved to be more stable in the presence of outliers.

Proceeding this way, $Y_{1},...,Y_{n}$ are
independent observations but they are not identically distributed. It can be easily seen that the
distribution of $Y_{i}$ is given by

\begin{equation}
f_{\alpha ,\beta }^{(i)}(y)=\frac{c\left( i,n\right) \text{ }\beta \left(
\frac{y_{i}}{\alpha }\right) ^{i\beta -1}}{\alpha \left( 1+\left( \frac{y_{i}%
}{\alpha }\right) ^{\beta }\right) ^{n+1}}\text{ }y>0,  \label{1}
\end{equation}
being

\[
c\left( i,n\right) = {n!\over (n-i)!(i-1)!}.
\]

Let us finally introduce the family of {\bf density power divergence} (DPD). This family of divergences has been introduced in (\cite{bahahjjo98}) and it has been applied in many different fields. Given two densities $f$ and $g$, the DPD measure, $d_{\tau }(f,g),$ between $g$ and $f$, as the function of a single tuning
parameter $\tau $ ($\geq 0$), is defined by

\begin{equation}
d_{\tau }(g,f)=\int \left \{ f^{1+\tau }(x)-\left( 1+\frac{1}{\tau }\right)
f^{\tau }(x)g(x)+\frac{1}{\tau }g^{1+\tau }(x)\right \} dx.  \label{10.1}
\end{equation}

Taking limits for $\tau \rightarrow 0,$ it follows that

\begin{equation*}
d_{0}(g,f)=\lim_{\tau \rightarrow 0}d_{\tau }(g,f)=\int g(x)\log f(x)dx,
\end{equation*}%
i.e. the classical Kullback-Leibler divergence (see (\cite{par06}) for more details about Kullback-Leibler divergence). The parameter $\tau $ controls
the trade-off between efficiency and robustness of the MDPDE. Hence, it has been shown in many different problems that DPD offers robust estimators in the presence of outliers at a reduced cost in terms of efficiency (see e.g. (\cite{bachghpa22, bacamapa19, bacamapa19b, pamabapa13, bajapa24})).

\subsection{Minimum density power divergence estimators}

Let us then turn to the problem of estimating the parameters of the log-logistic distribution when RSS is considered via DPD. We have seen that $Y_{1},...,Y_{n}$ are independent observations but they are
not identically distributed. Initially we shall assume that the
probability density function of $Y_{i}$ is $g_{i},$ $i=1,...,n,$ being $g_{1},...,g_{n}$ different probability density functions. We want to model $g_{i}$\ $i=1,...,n,$ by $\ f_{\alpha ,\beta }^{(i)}(y), i=1,...,n,$
where $f_{\alpha ,\beta }^{(i)}(y)$ was defined in (\ref{1}). Note that the distributions $f_{\alpha ,\beta }^{(i)}(y)$ are different for different values of $i$ but the unknown parameters $\boldsymbol{\theta }\mathbf{=}\left( \alpha ,\beta \right) $ are common for all of them.

Our interest is to estimate $\boldsymbol{\theta }%
\mathbf{=}\left( \alpha ,\beta \right) $ by minimizing the DPD between the
data and the model. However, here the model density is different for each
$Y_i$, and hence we need to calculate the divergence between data and model
separately for each data point. Considering all data points, it makes sense to choose as criterion to minimize the average divergence between the data points and the
models. Because the DPD measures the similarity between two distributions, the best parameter value $\boldsymbol{\theta}$ approximating the distribution $g_i$ should minimize the divergence between the true and assumed densities, $g_i$ and $f^{Y_i}_{\boldsymbol{\theta}}.$ Therefore, the minimum density power divergence functional at the distribution $g_i$, denoted by $\boldsymbol{T}_{\tau}(g_i)$, is defined as
\begin{equation}
d_{\tau }(g_i,f^{Y_i}_{\boldsymbol{T}_{\tau }(g_i)})=\min_{\boldsymbol{\theta }\in \boldsymbol{\Theta }} d_{\tau}(g_i,f^{Y_i}_{\boldsymbol{\theta }}).  \label{10.2}
\end{equation}%
where $\boldsymbol{T}_{\tau }(g_i) \in \boldsymbol{\Theta }$ denotes the best parameter so the distributions $g_i$ and $f^{Y_i}_{\boldsymbol{T}_{\tau }(g_i)}$ are as close as possible. Therefore, if $d_{\tau }(g_{i},f_{\alpha ,\beta }^{Y_{i}})$
denotes the DPD between the density estimate corresponding to the $i$-th data
point and the associated model density, we must minimize

\begin{equation}\label{MDPDE-2}
\frac{1}{n}\sum_{i=1}^{n}d_{\tau }(g_{i},f_{\alpha ,\beta
}^{(i)}).
\end{equation}

Moreover, the real distributions $g_i$ are in general unknown in practice. Hence, these distributions need to be estimated and as we only have one datum for this distribution, the estimation is given by the degenerate distribution at this point. We will denote this degenerate distribution by $\widehat{g}_{i}.$

We will call the {\bf minimum density power divergence estimators} (MDPDE), denoted $\hat{\alpha }_{\tau }, \hat{\beta }_{\tau },$ the values of $\alpha $ and $\beta $ respectively where the minimum of

\begin{equation}\label{MDPDE}
\frac{1}{n}\sum_{i=1}^{n}d_{\tau }(\widehat{g}_{i},f_{\alpha ,\beta
}^{(i)})
\end{equation}
is attained. Here, we are denoting $f^{Y_i}_{\alpha ,\beta
}$ by $f_{\alpha ,\beta
}^{(i)}.$

It can be easily seen (see e.g. (\cite{ghba13})) that in order to get the MDPDE
for $\boldsymbol{\theta }\mathbf{=}\left( \alpha ,\beta \right) $ for the objective function stated in (\ref{MDPDE}), it suffices to consider

\begin{equation}
H_{n,\tau }\left( \alpha ,\beta \right) =\frac{1}{n}\sum_{i=1}^{n}%
\left\{ \int_{0}^{\infty }f_{\alpha ,\beta }^{(i)}(y)^{\tau
+1}dy-\left( 1+\frac{1}{\tau }\right) f_{\alpha ,\beta }^{(i)}(y_i)^{\tau
}\right\} . \label{2}
\end{equation}

In the following theorem we get the expression of

\begin{equation}
\int_{0}^{\infty }f_{\alpha ,\beta }^{(i)}(y)^{\tau +1}dy.
\label{3}
\end{equation}

\begin{theorem}\label{densidad}
The expression of (\ref{3}) is given by

\[
\int_{0}^{\infty }f_{\alpha ,\beta }^{(i)}(y)^{\tau +1}dy=\frac{%
c\left( i,n\right) \beta ^{\tau }}{\alpha ^{\tau }}B\left( \frac{\left(
n+i-1\right) \tau \beta +\tau +\beta \left( n+i-1\right) }{\beta },\frac{%
i\beta \tau +i\beta -\tau }{\beta }\right) ,
\]%
where $B(a,b)$ is the Beta function with arguments $a$ and $b,$ i.e,
\begin{equation}\label{beta}
B\left( a,b\right) =\int_{0}^{1}x^{a-1}\left( 1-x\right) ^{b-1}dx.
\end{equation}
\end{theorem}

\begin{proof}
See Supplementary Material.
\end{proof}

Based on the previous theorem the objective function given in (\ref{2}) can
be written as

\[
H_{n,\tau }\left( \alpha ,\beta \right) = {1\over n} \sum_{i=1}^n \left\{ c(i,n)^{\tau +1} \left( {\beta \over \alpha }\right)^{\tau } B\left( \frac{\left(
n+i-1\right) \tau \beta +\tau +\beta \left( n+i-1\right) }{\beta },\frac{i\beta \tau +i\beta -\tau }{\beta }\right) \right.
\]

\[ \left. - \left( 1 + {1\over \tau }\right) {c(i,n)^{\tau } \beta^{\tau } \left( {y_i \over \alpha }\right)^{\tau (i\beta -1)} \over \alpha^{\tau } \left( 1+ \left( {y_i \over \alpha }\right)^{\beta } \right)^{\tau (n+1)}} \right\} \]

In what follows, we will denote by $\Psi (a)$ the digamma function, defined as the logarithmic derivative of the gamma function. Hence, the following holds.

\begin{theorem}\label{TheoEqMDPDE}
The MDPDE for the parameters $\alpha , \beta $ of the log-logisitc distribution when RSS is applied are given as the solutions of the system

\begin{eqnarray*}
0 & = & {1\over n} \sum_{i=1}^n \left\{ {c(i,n) \over \alpha^{\tau }} B\left( \frac{\left(
n+i-1\right) \tau \beta +\tau +\beta \left( n+i-1\right) }{\beta },\frac{i\beta \tau +i\beta -\tau }{\beta }\right) {(-\tau )\over \alpha} \right. \\
& & - \left. \left( 1 + {1\over \tau }\right) y_i^{\tau (i\beta -1)} {\alpha^{\beta \tau (n-i+1)} \over \left( \alpha^{\beta }+ y_i^{\beta } \right)^{\tau (n+1)}} \left[ {y_i^{\beta }\beta \tau (n-i+1) - \alpha^{\beta }\beta \tau i \over \alpha \left( \alpha^{\beta } + y_i^{\beta }\right) } \right] \right\} . \\
0 & = & \sum_{i=1}^n c(i,n) B\left( \frac{\left(
n+i-1\right) \tau \beta +\tau +\beta \left( n+i-1\right) }{\beta },\frac{i\beta \tau +i\beta -\tau }{\beta }\right) \\
& & \left[ 1 + {1\over \beta }\left( \Psi\left( \frac{i\beta \tau +i\beta -\tau }{\beta } \right) - \Psi \left( \frac{\left(
n+i-1\right) \tau \beta +\tau +\beta \left( n+i-1\right) }{\beta }\right) \right) \right] \\
& & - \left( 1 + {1\over \tau }\right) \beta \left( {y_i\over \alpha }\right)^{\tau (i\beta -1)} \\
& & \left[ \left( {1\over \beta } + i \ln \left( {y_i\over \alpha }\right) \right) \left( 1 + \left( {y_i\over \alpha }\right)^{\beta } \right)^{\tau (n -1)} - (n+1) \left( {y_i\over \alpha }\right)^{\beta } \ln \left( {y_i\over \alpha }\right) \left( 1 + \left( {y_i\over \alpha }\right)^{\beta } \right)^{\tau (n -1)-1} \right] .
\end{eqnarray*}
\end{theorem}

\begin{proof}
See Supplementary Material.
\end{proof}

\section{Asymptotic distribution of the MDPDE for the log-logistic model
for the RSS}

In this section, we derive the asymptotic distribution of MDPDE for the log-logisitc model assuming RSS in three different situations. First, we assume that $\beta $ is known and only $\alpha $ needs to be estimated. Next, we consider the case in which $\alpha $ is known and it is parameter $\beta $ that needs to be estimated. We finally treat the case in which both parameters need to be estimated.

To obtain the asymptotic distribution, we use the results in (\cite{ghba13}) in which  it is shown that under mild conditions, given $Y_1, ..., Y_n$ independent and not identically distributed variables
all of them depending on a fixed parameter $\bm \theta ,$ then the MDPDE  $\widehat{\bm \theta }_{\tau }$ of $\bm \theta $ satisfies

\begin{equation}\label{DistAsintotica}
\sqrt{n}\left( \widehat{\bm \theta }_{\tau }-\bm \theta \right) \underset{%
n\longrightarrow \infty }{\overset{\mathcal{L}}{\longrightarrow }}\mathcal{N}%
\left( 0,J_{\tau }^{-1}(\bm \theta )K_{\tau }(\bm \theta )J_{\tau }^{-1}(\bm \theta
)\right) ,
\end{equation}
being
\begin{equation}\label{MatrizJ}
J_{\tau }(\bm \theta )={1\over n} \sum_{i=1}^n J_{\tau }^{(i)}(\bm \theta ) \quad \text{
and }\quad K_{\tau }(\bm \theta )=J_{2\tau }(\bm \theta )-\xi _{\tau }\left( \bm \theta
\right) \xi _{\tau }\left( \bm \theta
\right)^T,
\end{equation}
and

\begin{equation}\label{Vectorxi}
\xi _{\tau }\left( \bm \theta \right) = {1\over n} \sum_{i=1}^n \xi _{\tau }^{(i)}\left( \bm \theta \right) .
\end{equation}

Here, $J_{\tau }^{(i)}(\bm \theta )$ and $\xi _{\tau }^{(i)}\left( \bm \theta \right)$ refer to the corresponding values for $Y_i.$ Hence,

$$ J_{\tau }^{(i)}(\bm \theta )= \int_{\mathbb{R}} \left( {\partial log f_{\bm \theta }^{\bm Y_i} (x) \over \partial \bm \theta }\right)^2 f_{\bm \theta }^{\bm Y_i} (x)^{\tau +1} dx,$$
and

$$ \xi _{\tau }^{(i)} \left( \bm \theta \right) =\int_{\mathbb{R}} \frac{\partial
\log f_{\bm \theta }^{\bm Y_i}(x)}{\partial \bm \theta }f_{\bm \theta }^{\bm Y_i}(x)^{\tau +1}dx.$$

\subsection{Asymptotic distribution of $\protect\widehat{\protect\alpha }_{%
\protect\tau }$}

In relation to the asymptotic distribution of $\widehat{\alpha }_{\tau }$,
assuming $\beta $ known, we have applying the previous results for $\bm \theta = \alpha $

\begin{equation*}
\sqrt{n}\left( \widehat{\alpha }_{\tau }-\alpha \right) \underset{%
n\longrightarrow \infty }{\overset{\mathcal{L}}{\longrightarrow }}\mathcal{N}%
\left( 0,J_{\tau }^{-1}(\alpha )K_{\tau }(\alpha )J_{\tau }^{-1}(\alpha
)\right) ,
\end{equation*}
being $J_{\tau }(\alpha )$ and $\xi _{\tau }\left( \alpha \right) $ the corresponding values given in Eqs. (\ref{MatrizJ}) and (\ref{Vectorxi}), repectively. In our case, $J_{\tau }^{(i)}(\alpha )$ and $\xi _{\tau }^{(i)}\left( \alpha \right)$ are referred to the corresponding values for the $i$-th ordered statistic of the log-logistic distribution. Then,

$$ J_{\tau }^{(i)}(\alpha )= \int_0^{\infty } \left( {\partial log f_{\alpha }^{(i)} (x) \over \partial \alpha }\right)^2 f_{\alpha }^{(i)} (x)^{\tau +1} dx,\quad \xi _{\tau }^{(i)} \left( \alpha \right) =\int_{0}^{\infty }\frac{\partial
\log f_{\alpha }^{(i)}(x)}{\partial \alpha }f_{\alpha }^{(i)}(x)^{\tau +1}dx.$$

In the next theorem we are going to get the expression of $J_{\tau }^{(i)}(\alpha
).$

\begin{theorem}
\label{Theorem1}For the log-logistic distribution given in (\ref{densidad}) we
have

\begin{equation*}
J_{\tau }^{(i)}(\alpha )=
A_1 + A_2 + A_3,
\end{equation*}
where $A_1, A_2, A_3$ are given by

\begin{eqnarray*}
A_1 & = & \left({\beta \over \alpha }\right)^{\tau +2} (n-i+1)^2 c(i,n)^{\tau +1} B\left( {(n-i+1) \tau \beta + (n-i+1) \beta +\tau \over \beta }, {i \tau \beta + i \beta -\tau \over \beta }\right) .\\
A_2 & = & \left({\beta \over \alpha }\right)^{\tau +2} (n-i+1)^2 c(i,n)^{\tau +1} \cdot {(n-i+1) \tau \beta + (n-i+2) \beta +\tau \over \beta \left[ (n+1) \tau + (n+2) \right] } \\
& & \times {(n-i+1) \tau \beta + (n-i+1) \beta +\tau \over \beta \left[ (n+1) \tau + (n+1) \right] }  B\left( {(n-i+1) \tau \beta + (n-i+1) \beta +\tau \over \beta }, {i \tau \beta + i \beta -\tau \over \beta }\right) . \\
A_3 & = & \left({\beta \over \alpha }\right)^{\tau +2} 2(n-i+1) (n+1) c(i,n)^{\tau +1} \cdot {(n-i+1) \tau \beta + (n-i+1) \beta +\tau \over (n+1) \tau \beta + (n+1) \beta } \\
& & B\left( {(n-i+1) \tau \beta + (n-i+1) \beta +\tau \over \beta }, {i \tau \beta + i \beta -\tau \over \beta }\right) .
\end{eqnarray*}
\end{theorem}

\begin{proof}
See Supplementary Material.
\end{proof}

\begin{corollary}\label{coralpha}
The Fisher information for $\alpha $ is given by%
\begin{equation*}
I_{F}^{(i)}\left( \alpha \right) =J_{\tau =0}^{(i)}(\alpha )=\left( \frac{\beta }{\alpha }\right)
^{2}(1+\frac{4}{3}-2)=\frac{\beta ^{2}}{3\alpha ^{2}}.
\end{equation*}
\end{corollary}

\begin{proof}
See Supplementary Material.
\end{proof}

Let us now turn to matrix $K_{\tau }^{(i)}(\alpha ).$

\begin{theorem}
\label{Theorem1a}For the log-logistic distribution given in (\ref{densidad}) we
have,
\begin{equation*}
K_{\tau }^{(i)}(\alpha )=J_{2\tau }^{(i)}(\alpha )-\xi _{\tau }^{(i)}\left( \alpha \right)
^{2},
\end{equation*}%
being
\begin{equation*}
\xi_{\tau }^{(i)}\left( \alpha \right) = \left(
\frac{\beta }{\alpha }\right) ^{\tau +1}B\left( \frac{(n-i+1)\tau \beta +(n-i+1)\beta
+\tau }{\beta },\frac{i\tau \beta +i\beta -\tau }{\beta }\right) \left( \frac{%
-\tau }{\beta +\tau \beta }\right) .
\end{equation*}
\end{theorem}

\begin{proof}
See Supplementary Material.
\end{proof}

\begin{corollary}
We can observe that for $\tau =0$ we have
\begin{equation*}
\xi_{\tau =0}^{(i)}\left( \alpha \right) =0\text{ and }K_{\tau =0}^{(i)}(\alpha
)=I_{F}^{(i)}\left( \alpha \right) .
\end{equation*}
\end{corollary}

\subsection{Asymptotic distribution of $\protect\widehat{\protect\beta }%
\protect\tau $}

In relation to the asymptotic distribution of $\widehat{\beta }_{\tau }$,
assuming $\alpha $ known, we have applying Eq. (\ref{DistAsintotica})

\begin{equation*}
\sqrt{n}\left( \widehat{\beta }_{\tau }-\beta \right) \underset{%
n\longrightarrow \infty }{\overset{\mathcal{L}}{\longrightarrow }}\mathcal{N}%
\left( 0,J_{\tau }^{-1}(\beta )K_{\tau }(\beta )J_{\tau }^{-1}(\beta
)\right) .
\end{equation*}
%
%

Hence, it just suffices to obtain $$ J_{\tau }^{(i)}(\beta )= \int_0^{\infty } \left( {\partial log f_{\beta }^{(i)} (x) \over \partial \beta }\right)^2 f_{\beta }^{(i)} (x)^{\tau +1} dx,\quad \xi _{\tau }^{(i)} \left( \beta \right) =\int_{0}^{\infty }\frac{\partial
\log f_{\beta }^{(i)}(x)}{\partial \beta }f_{\beta }^{(i)}(x)^{\tau +1}dx.$$

In the next theorem we are going to get the expression of $J_{\tau }^{(i)}(\beta ).$

\begin{theorem}
\label{Theorem2}For the log-logistic distribution given in (\ref{densidad}) we
have

\begin{equation*}
J_{\tau }^{(i)}(\beta )
= C_1 + C_2 + C_3 + C_4 - C_5 - C_6,
\end{equation*}
where $C_1, C_2, C_3, C_4, C_5, C_6$ are given by

{\small
\begin{eqnarray*}
C_1 & = & \left({\beta \over \alpha }\right)^{\tau } {c(i,n)^{\tau +1}\over \beta^2} B\left( {(n-i+1) \tau \beta + (n-i+1) \beta +\tau \over \beta }, {i \tau \beta + i \beta -\tau \over \beta }\right) .\\
C_2 & = & \left({\beta \over \alpha }\right)^{\tau } {c(i,n)^{\tau +1}\over \beta^2} \cdot i^2 B\left( {(n-i+1) \tau \beta + (n-i+1) \beta +\tau \over \beta }, {i \tau \beta + i \beta -\tau \over \beta }\right) \\
& &  \times
\left[ \Psi' \left( {i \tau \beta + i \beta -\tau \over \beta }\right)  + \Psi' \left( {(n-i+1) \tau \beta + (n-i+1) \beta +\tau \over \beta }\right) \right. \\
 & & \left.
 + \left( \Psi \left( {i \tau \beta + i \beta -\tau \over \beta }\right) - \Psi \left( {(n-i+1) \tau \beta + (n-i+1) \beta +\tau \over \beta }\right) \right)^2 \right] .\\
C_3 & = & \left({\beta \over \alpha }\right)^{\tau } {c(i,n)^{\tau +1}\over \beta^2} \cdot (n+1)^2 \\
& & \times B\left( {(n-i+1) \tau \beta + (n-i+1) \beta +\tau \over \beta }, {i \tau \beta + (i+2) \beta -\tau \over \beta }\right) \\
& & 
\times \left[ \Psi' \left( {i \tau \beta + (i+2) \beta -\tau \over \beta }\right) . + \Psi' \left( {(n-i+1) \tau \beta + (n-i+1) \beta +\tau \over \beta }\right) \right. \\
 & & \left. + \left( \Psi \left( {i \tau \beta + (i+2) \beta -\tau \over \beta }\right) - \Psi \left( {(n-i+1) \tau \beta + (n-i+1) \beta +\tau \over \beta }\right) \right)^2 \right] .\\
C_4 & = & \left({\beta \over \alpha }\right)^{\tau } {c(i,n)^{\tau +1}\over \beta^2} \cdot 2i B\left( {(n-i+1) \tau \beta + (n-i+1) \beta +\tau \over \beta }, {i \tau \beta + i \beta -\tau \over \beta }\right) \\
 & & \left[ \Psi \left( {i \tau \beta + i \beta -\tau \over \beta }\right) - \Psi \left( {(n-i+1) \tau \beta + (n-i+1) \beta +\tau \over \beta }\right) \right] .\\
C_5 & = & \left({\beta \over \alpha }\right)^{\tau } {c(i,n)^{\tau +1}\over \beta^2} \cdot 2(n+1) B\left( {(n-i+1) \tau \beta + (n-i+1) \beta +\tau \over \beta }, {i \tau \beta + (i+1) \beta -\tau \over \beta }\right) \\
 & & \left[ \Psi \left( {i \tau \beta + (i+1) \beta -\tau \over \beta }\right) - \Psi \left( {(n-i+1) \tau \beta + (n-i+1) \beta +\tau \over \beta }\right) \right] .\\
C_6 & = & \left({\beta \over \alpha }\right)^{\tau } {c(i,n)^{\tau +1}\over \beta^2} \cdot 2i(n+1) B\left( {(n-i+1) \tau \beta + (n-i+1) \beta +\tau \over \beta }, {i \tau \beta + (i+1) \beta -\tau \over \beta }\right) \\
& & \left[ \Psi' \left( {i \tau \beta + (i+1) \beta -\tau \over \beta }\right)   + \Psi' \left( {(n-i+1) \tau \beta + (n-i+1) \beta +\tau \over \beta }\right) \right. \\
& & \left. + \left( \Psi \left( {i \tau \beta + (i+1) \beta -\tau \over \beta }\right) - \Psi \left( {(n-i+1) \tau \beta + (n-i+1) \beta +\tau \over \beta }\right) \right)^2 \right] .\\
\end{eqnarray*}}
\end{theorem}

\begin{proof}
See Supplementary Material.
\end{proof}

\begin{theorem}
\label{Theorem2a}For the log-logistic distribution given in (\ref{densidad}) we
have,
\begin{equation*}
K_{\tau }^{(i)}(\beta )=J_{2\tau }^{(i)}(\beta )-\xi_{\tau }^{(i)}\left( \beta \right)
^{2},
\end{equation*}%
being
\begin{eqnarray*}
\xi_{\tau }^{(i)}\left( \beta \right) & = & \left(
\frac{\beta }{\alpha }\right) ^{\tau } {c(i,n)^{\tau +1} \over \beta }B\left( \frac{(n-i+1)\tau \beta +(n-i+1)\beta
+\tau }{\beta },\frac{i\tau \beta +i\beta -\tau }{\beta }\right) \\
& & \frac{\tau }{1 +\tau } \cdot \left[ {1\over \beta } \left( \Psi \left(\frac{i\tau \beta +i\beta -\tau }{\beta }\right) -\Psi \left( \frac{(n-i+1)\tau \beta +(n-i+1)\beta
+\tau }{\beta }\right)  \right) +1 \right] .
\end{eqnarray*}
\end{theorem}

\begin{proof}
See Supplementary Material.
\end{proof}

\subsection{Asymptotic distribution of $\left( \protect\widehat{\protect%
\alpha }_{\protect\tau },\protect\widehat{\protect\beta }\protect\tau %
\right) $}

In relation to the asymptotic distribution of $\left( \widehat{\alpha }%
_{\tau },\widehat{\beta }_{\tau }\right) $ , we have,
\begin{equation*}
\sqrt{n}\left( \left( \widehat{\alpha }_{\tau },\widehat{\beta }_{\tau
}\right) ^{T}-\left( \alpha ,\beta \right) ^{T}\right) \underset{%
n\longrightarrow \infty }{\overset{\mathcal{L}}{\longrightarrow }}\mathcal{N}%
\left( \boldsymbol{0},\boldsymbol{J}_{\tau }^{-1}(\alpha ,\beta )\boldsymbol{%
K}_{\tau }(\alpha ,\beta )\boldsymbol{J}_{\tau }^{-1}(\alpha ,\beta )\right)
\end{equation*}%
being%
\begin{equation*}
\boldsymbol{J}_{\tau }(\alpha ,\beta )=\left(
\begin{array}{cc}
J_{\tau }^{11}\left( \alpha ,\beta \right) & J_{\tau }^{12}\left( \alpha
,\beta \right) \\
J_{\tau }^{12}\left( \alpha ,\beta \right) & J_{\tau }^{22}\left( \alpha
,\beta \right)%
\end{array}%
\right)
\end{equation*}%
and
\begin{equation*}
\boldsymbol{K}_{\tau }(\alpha ,\beta )=\boldsymbol{J}_{2\tau }(\alpha ,\beta
)-\xi _{\tau }\left( \alpha ,\beta \right) ^{T}\xi _{\tau }\left( \alpha
,\beta \right) .
\end{equation*}%

Here

\begin{equation*}
J_{\tau }^{11}\left( \alpha ,\beta \right) =J_{\tau }\left( \alpha \right) ,%
\text{ }J_{\tau }^{22}\left( \alpha ,\beta \right) =J_{\tau }\left( \beta
\right) ,\text{ }J_{\tau }^{12}\left( \alpha ,\beta \right)
={1\over n} \sum_{i=1}^n J_{\tau }^{(i)12}\left( \alpha ,\beta \right)
\end{equation*}%
and

\begin{equation*}
\xi _{\tau }\left( \alpha ,\beta \right) = {1\over n} \sum_{i=1}^n \xi _{\tau }^{(i)}\left( \alpha ,\beta \right) .
\end{equation*}

Finally,

\begin{equation*}
J_{\tau }^{(i)12}\left( \alpha ,\beta \right) =\int_{0}^{\infty
}\left( \frac{\partial \log f_{\alpha ,\beta }(x)^{(i)}}{\partial \alpha }\right)
\left( \frac{\partial \log f_{\alpha ,\beta }^{(i)}(x)}{\partial \beta }\right)
f_{\alpha , \beta }^{(i)}(x)^{\tau +1}dx
\end{equation*}

and

\begin{eqnarray*}
\xi _{\tau }^{(i)}\left( \alpha ,\beta \right) = \left( \xi _{\tau }^{(i)}\left( \alpha
\right) ,\xi _{\tau }^{(i)}\left( \beta \right) \right)^T .
\end{eqnarray*}

Thus, we just need to derive the expression for $J_{\tau }^{(i)12}\left( \alpha ,\beta \right)$. This is done in next theorem.

\begin{theorem}
\label{Theorem4}For the log-logistic distribution given in (\ref{densidad}) we
have%

%
\begin{equation*}
J_{\tau }^{(i)12}\left( \alpha ,\beta \right) =E_{1}+E_{2}-E_{3}-E_{4}-E_{5}+E_{6},
\end{equation*}%
where $E_{1}, E_{2}, E_{3}, E_{4}, E_{5}$ and $E_{6}$ are given by

\begin{eqnarray*}
E_{1} & = & \frac{\beta ^{\tau }}{\alpha ^{\tau +1}}B\left( \frac{\tau \beta
+\beta +\tau }{\beta },\frac{\tau \beta +\beta -\tau }{\beta }\right) .\\
E_{2} & = & \frac{\beta ^{\tau }}{\alpha ^{\tau +1}}B\left( \frac{\tau \beta
+\beta +\tau }{\beta },\frac{\tau \beta +\beta -\tau }{\beta }\right)
\left\{ \Psi \left( \frac{\tau \beta +\beta -\tau }{\beta }\right) -\Psi
\left( \frac{\tau \beta +\beta +\tau }{\beta }\right) \right\} .\\
E_{3} & = & -2\frac{\beta ^{\tau }}{\alpha ^{\tau +1}}B\left( \frac{\tau \beta
+\beta +\tau }{\beta },\frac{\tau \beta +2\beta -\tau }{\beta }\right)
\left\{ \Psi \left( \frac{\tau \beta +2\beta -\tau }{\beta }\right) -\Psi
\left( \frac{\tau \beta +\beta +\tau }{\beta }\right) \right\} .\\
E_{4} & = & -2\frac{\beta ^{\tau }}{\alpha ^{\tau +1}}B\left( \frac{\tau \beta
+2\beta +\tau }{\beta },\frac{\tau \beta +\beta -\tau }{\beta }\right) .\\
E_{5} & = & -2\frac{\beta ^{\tau }}{\alpha ^{\tau +1}}B\left( \frac{\tau \beta
+2\beta +\tau }{\beta },\frac{\tau \beta +\beta -\tau }{\beta }\right)
\left\{ \Psi \left( \frac{\tau \beta +\beta -\tau }{\beta }\right) -\Psi
\left( \frac{\tau \beta +2\beta +\tau }{\beta }\right) \right\} .\\
E_{6} & = & 4\frac{\beta ^{\tau }}{\alpha ^{\tau +1}}B\left( \frac{\tau \beta
+2\beta +\tau }{\beta },\frac{\tau \beta +2\beta -\tau }{\beta }\right)
\left\{ \Psi \left( \frac{\tau \beta +2\beta -\tau }{\beta }\right) -\Psi
\left( \frac{\tau \beta +2\beta +\tau }{\beta }\right) \right\} .
\end{eqnarray*}
\end{theorem}

\begin{proof}
See Supplementary Material.
\end{proof}

For $\tau =0,$ we obtain the following result:

\begin{corollary}
\label{cor4}For the log-logistic distribution given in (\ref{densidad}) we
have%

\begin{equation*}
J_{0}^{(i)}\left( \alpha ,\beta \right)
=\left( \begin{array}{cc} J_{0}^{(i)11}\left( \alpha , \beta \right) & J_{0}^{(i)12}\left( \alpha ,\beta \right) \\ J_{\tau }^{(i)12}\left( \alpha ,\beta \right) & J_{0}^{(i)22}\left( \alpha , \beta \right) \end{array}, \right)
\end{equation*}
where $J_{0}^{(i)11}\left( \alpha , \beta \right) = J_{0}^{(i)}\left( \alpha \right)$ and $J_{0}^{(i)22}\left( \alpha , \beta \right) = J_{0}^{(i)}\left( \beta \right)$ has been obtained in (\cite{hechqi20}) and

\begin{equation*}
J_{0}^{(i)12}\left( \alpha ,\beta \right)
={1\over \alpha } {1\over n+2} \left[ n-2i+1-i(n-i+1)(\Psi (i) - \Psi (n-i+1))\right] .
\end{equation*}%
\end{corollary}

\begin{remark}
Note that

\begin{equation*}
\sum_{i=1}^n i(n-i+1)(\Psi (i) - \Psi (n-i+1)) =0.
\end{equation*}

On the other hand,

\begin{equation*}
\sum_{i=1}^n (n-2i+1) = n(n+1) - 2 {n(n+1)\over 2} = 0.
\end{equation*}

Consequently,

\begin{equation*}
\sum_{i=1}^n J_{0}^{(i)12}\left( \alpha ,\beta \right)= 0.
\end{equation*}
\end{remark}

\section{Numerical analysis}



To assess the performance of MDPDE in estimating log-logistic parameters for data coming from a rank sample, we have conducted a simulation study. In this study, we compare MDPDE coming from RSS for different values of parameters $\alpha $ and $\beta $ and for different sample sizes. We consider that $\alpha $ is fixed to 1, while $\beta $ attains values $\beta =1.5, 2.5, 5.0, 10.$ The sample sizes are $n=10, 25, 50, 75,
100.$ Hence, for a concrete value of $n$, we generate $n$ samples of size $n$. For the first sample, denoted $(X_1^1, ..., X_n^1),$ we define $Y_1= X_{(1)}^1.$ Similarly, for the $i$-th sample $(X_1^i, ..., X_n^i),$ we define $Y_i=X_{(i)}^i.$ Finally, we consider the sample $(Y_1, ..., Y_n).$

For each combination of model parameters and sample size, we estimate the log-logistic parameters using MLE and MDPDE approaches for sample $(Y_1, ..., Y_n)$, with different values of the tuning parameter, namely
$\tau = 0.1, 0.2, ..., 1.0.$ We have also considered the repeated median estimator (RM), the sample median estimator (SM), and the Hodge-Lehmann and Shamos estimator (HL) (see (\cite{mawapa23})) for samples coming from RSS.

To avoid any possible influence of the sample, we repeat this process for each combination $M= 1000$ times. To measure the performance of the different methods, we compute for the $i$-th sample

\begin{equation*}
Bias_i = |\hat{\alpha}_i - \alpha | + |\hat{\beta }_i - \beta | ,\, MSE_i = (\hat{\alpha}_i - \alpha )^2 + (\hat{\beta }_i - \beta )^2,
\end{equation*}
and we compute the average over the $M=1000$ simulations, i.e.

\begin{equation*}
Bias = {\frac{1}{M}} \sum_{i=1}^M Bias_i ,\, RMSE = \sqrt{{\frac{1}{M}}
\sum_{i=1}^M MSE_i}.
\end{equation*}

In a first step, we consider an uncontaminated scenario. The corresponding error measures can be found in the Supplementary Material. We include as an example the values
for $n=100, \alpha = 1, \beta = 5$ in Table 1.

\begin{table}
\begin{center}\label{n100Beta5}
\begin{tabular}{|c|cccc|}
\hline
& Bias & RMSE & $\hat{\alpha }$ & $\hat{\beta }$ \\ \hline
MLE         & 0.10319 & 0.12446 & 0.99999 & 5.00754 \\
$DPD_{0.1}$ & 0.09279 & 0.11053 & 1.00000 & 5.00540 \\
$DPD_{0.2}$ & 0.09203 & 0.10901 & 1.00000 & 5.00537 \\
$DPD_{0.3}$ & 0.09566 & 0.11314 & 0.99998 & 5.00622 \\
$DPD_{0.4}$ & 0.10117 & 0.11936 & 0.99996 & 5.00744 \\
$DPD_{0.5}$ & 0.10703 & 0.12607 & 0.99994 & 5.00881 \\
$DPD_{0.6}$ & 0.11268 & 0.13261 & 0.99992 & 5.01023 \\
$DPD_{0.7}$ & 0.11787 & 0.13873 & 0.99989 & 5.01164 \\
$DPD_{0.8}$ & 0.12253 & 0.14438 & 0.99987 & 5.01301 \\
$DPD_{0.9}$ & 0.12683 & 0.14957 & 0.99985 & 5.01435 \\
$DPD_{1.0}$ & 0.13092 & 0.15435 & 0.99983 & 5.01564 \\
RM         & 0.11775 & 0.14018 & 0.99973 & 4.92095 \\
SM         & 1.98093 & 1.97550 & 1.00388 & 3.03025 \\
HL         & 2.11494 & 2.11189 & 1.00008 & 2.88908 \\
\hline
\end{tabular}%
\end{center}
\caption{Results for $n=100$ and $\beta = 5.0.$}
\end{table}


As expected, MLE performs very well when there is no contamination and this performance improves with the sample size. For MDPDE, note that these estimators follow the same behavior as MLE. Moreover, the performance approaches to that of MLE when the sample size increases. Moreover, for small values of $\tau ,$ the behavior is more or less the same as MLE. Hence, except for $n=10,$ it can be seen that they are almost equivalent for $\tau <0.6.$ This was expected due to the limit properties of MDPDE. Compared to other estimators, the only one competing with MDPDE is RM, but it is overcome for small values of $\tau $ ($\leq 0.4$) in all cases.

It should be noted, as remarked in the introduction, that RMSE when using a sample of size $n$ with RSS is smaller than the corresponding value for a random sample of the same size. To see the extent of this reduction for the log-logistic distribution, we include in Table 2 the corresponding values when random sampling is used for $n=100, \alpha = 1, \beta = 5.$ Compare these values with the ones appearing in Table 1. For other situations, we refer to the tables in the Supplementary Material of this paper and the Supplementary Material of (\cite{fejamipa24}).

\begin{table}
\begin{center}
\begin{tabular}{|c|cccc|}
\hline
& Bias & RMSE & $\hat{\alpha }$ & $\hat{\beta }$ \\ \hline
MLE & 0.36844 & 0.43431 & 1.00030 & 5.06549 \\
$DPD_{0.1}$ & 0.37232 & 0.44030 & 1.00024 & 5.06584 \\
$DPD_{0.2}$ & 0.38300 & 0.45529 & 1.00016 & 5.07067 \\
$DPD_{0.3}$ & 0.39717 & 0.47472 & 1.00007 & 5.07793 \\
$DPD_{0.4}$ & 0.41270 & 0.49589 & 0.99999 & 5.08649 \\
$DPD_{0.5}$ & 0.42850 & 0.51711 & 0.99990 & 5.09562 \\
$DPD_{0.6}$ & 0.44379 & 0.53734 & 0.99982 & 5.10485 \\
$DPD_{0.7}$ & 0.45786 & 0.55603 & 0.99975 & 5.11385 \\
$DPD_{0.8}$ & 0.47062 & 0.57299 & 0.99968 & 5.12247 \\
$DPD_{0.9}$ & 0.48213 & 0.58821 & 0.99962 & 5.13062 \\
$DPD_{1.0}$ & 0.49251 & 0.60182 & 0.99957 & 5.13828 \\
RM & 0.40334 & 0.47948 & 0.99994 & 5.01636 \\
SM & 1.93207 & 1.93817 & 1.00447 & 3.10069 \\
HL & 2.09938 & 2.08758 & 1.00038 & 2.92855 \\ \hline
\end{tabular}%
\end{center}
\par
\label{tablasimsincont100-3}
\caption{Results for $n=100$ and $\protect\beta = 5.$}
\end{table}

In a second step we study the behavior of these estimators in the presence of contamination for ranked samples. For this, we repeat the previous study assuming $\alpha=1$, $\beta=5$, and $n=100$ under four different conditions of contamination:

\begin{itemize}
\item \textbf{Case 1:} For a fixed percentage of data $p,$ we change the value in the
sample for a new value coming from a log-logistic distribution with $\alpha
=1, \beta =0.2.$

\item \textbf{Case 2:} For a fixed percentage of data $p,$ we change the value in the
sample for a new value coming from a log-logistic distribution with $\alpha
=4, \beta =10.$

\item \textbf{Case 3:} For a fixed percentage of data $p,$ we change the value in the
sample for a new value coming from a uniform distribution $\mathcal{U}%
(0,20). $

\item \textbf{Case 4:} For a fixed percentage of data $p,$ we change the value in the
sample for a new value of 50.
\end{itemize}

The percentage of contamination ranges from no contamination ($p=0\% $) to $p=40\% ,$ increasing the contamination $5\% $ at each step. We use RMSE to assess the goodness-of-fit for each estimator. And we repeat this procedure $M=1000$ times to avoid the possible influence of a specific sample. In Figures \ref{fig:Caso1} to \ref{fig:Caso4}, the evolution of RMSE for MLE, $MDPDE_{0.3}, MDPDE_{0.5}, MDPDE_{0.8}$ and RM is shown for each case. Tables with the corresponding values can be found in the Supplementary Material.

\begin{figure}
	\centering
	\includegraphics[height=8cm, width=13cm]{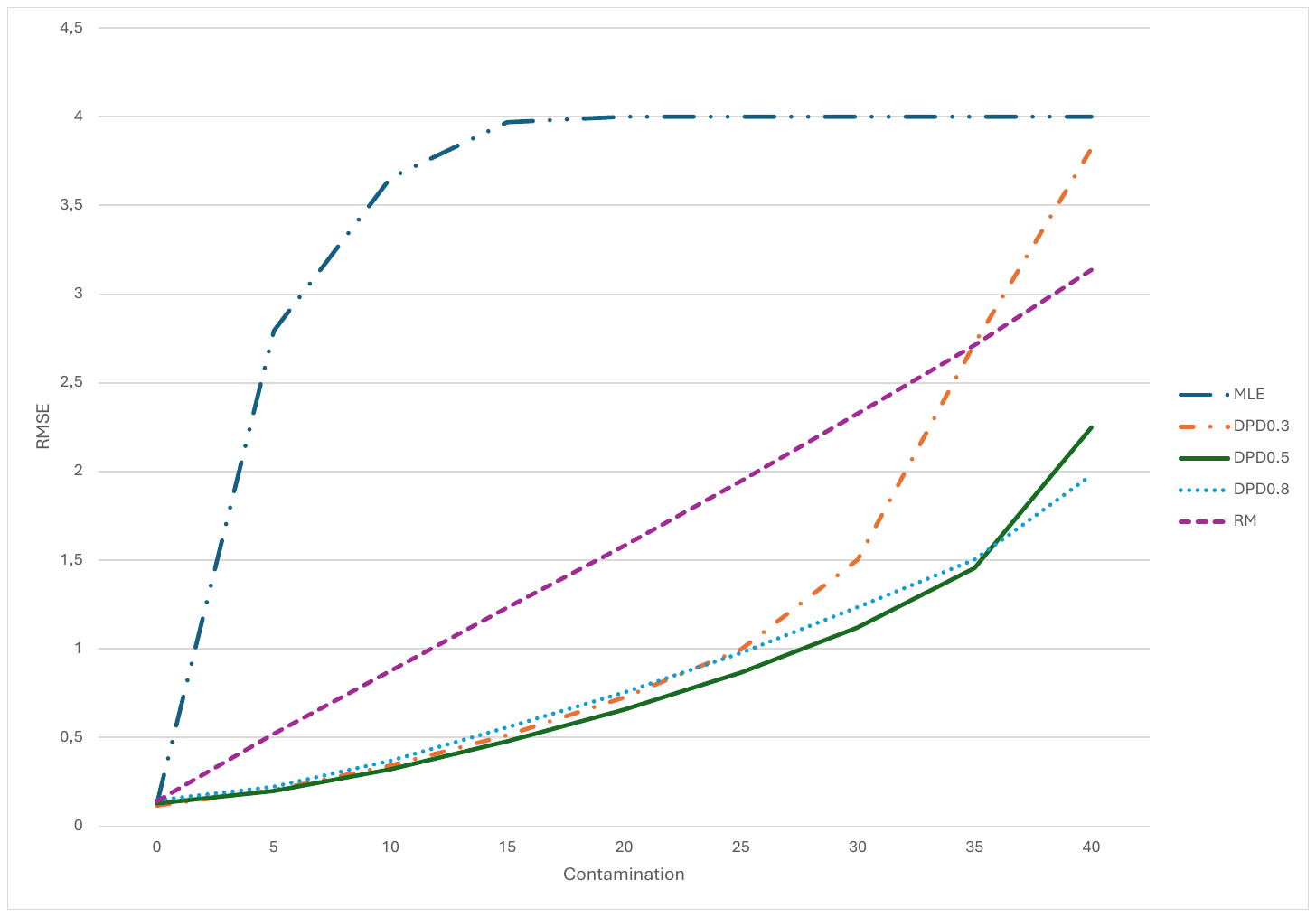}
	\caption{Evolution of RMSE for different values of contamination (Case 1).}
	\label{fig:Caso1}
\end{figure}

\begin{figure}
	\centering
	\includegraphics[height=8cm, width=13cm]{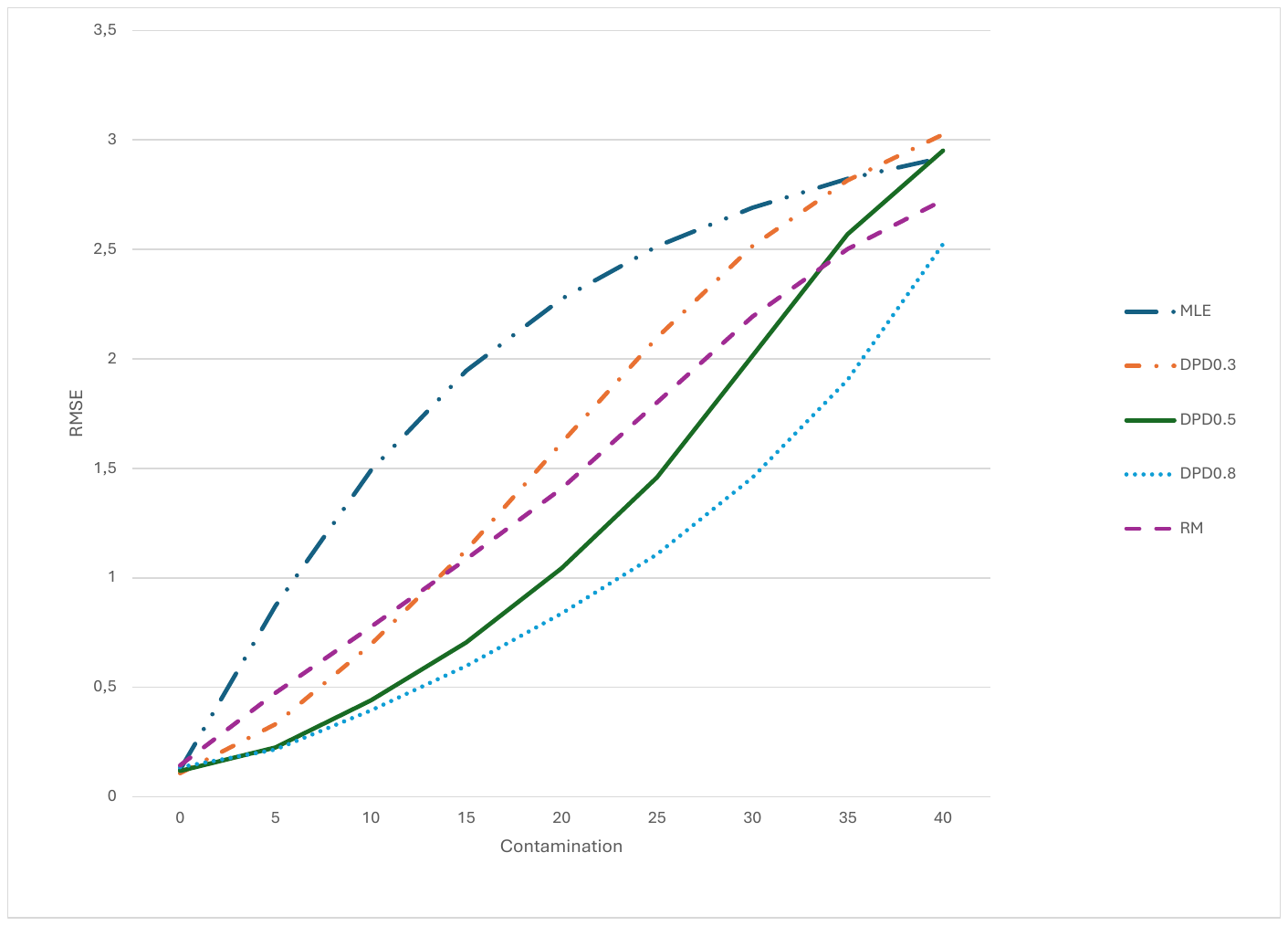}
	\caption{Evolution of RMSE for different values of contamination (Case 2).}
	\label{fig:Caso2}
\end{figure}

\begin{figure}
	\centering
	\includegraphics[height=8cm, width=13cm]{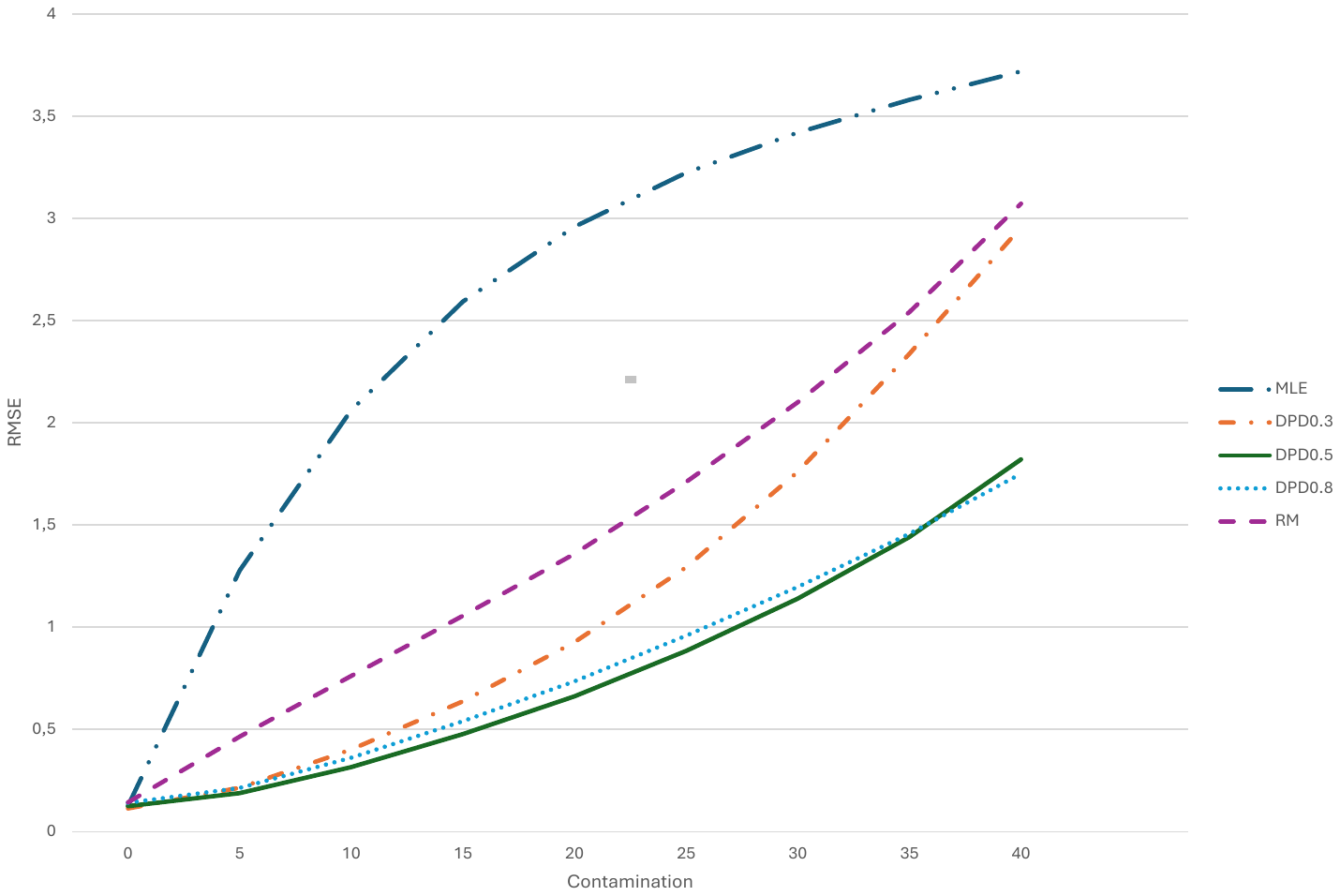}
	\caption{Evolution of RMSE for different values of contamination (Case 3).}
	\label{fig:Caso3}
\end{figure}

\begin{figure}
	\centering
	\includegraphics[height=8cm, width=13cm]{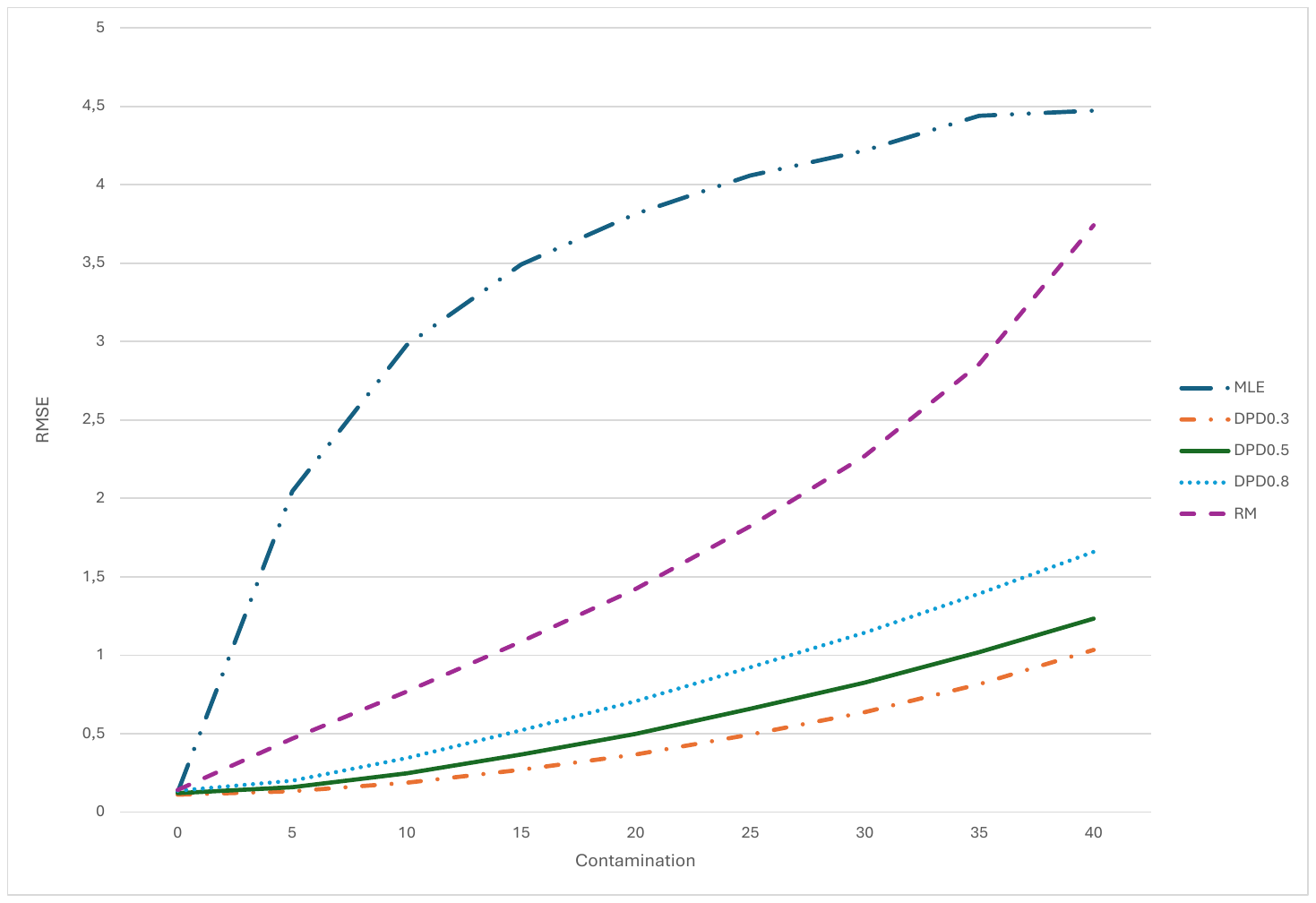}
	\caption{Evolution of RMSE for different values of contamination (Case 4).}
	\label{fig:Caso4}
\end{figure}

%
%
%
%
%

The results illustrate that while the MLE performs very well in the absence of contamination, its performance decreases sharply in the presence of contamination, and this happens in all the scenarios considered in the simulations. However, if we turn to MDPDE, we can see that they are competitive (indeed, we could say almost equivalent) in the absence of contamination for small values of $\tau $, while they provide very stable results when contamination arises and are very little affected by this contaminated data. And again, this happens in all considered scenarios. Comparing the results of MDPDE with the results provided with other estimators, we can see that only RM provides acceptable results, but it is outperformed by MDPDE for $\tau \in [0.2, 0.6].$

In conclusion, it seems at the light of this simulation study that DPD is an appealing alternative to MLE and other estimation methods, as its behavior is competitive with MLE in the absence of contamination and exhibits high stability when contamination occurs.

We finish this section with a note about the optimal value for the tuning parameter $\tau .$ As shown in the development of DPD and supported by the simulation study, small values of $\tau $ provide very efficient estimators while large values of $\tau $ provide robust estimators. For the log-logistic model using RSS, it seems that the best trade-off between efficiency and robustness appears for $\tau \in [0.2, 0.5].$

\section{Conclusions and open problems}

In this paper we have introduced a new family of estimators for the parameters of a log-logistic distribution considering RSS based on MDPDE. The choice of this family of divergences was justified by the fact that it has been previously used in other fields leading to robust estimators and besides, it includes MLE. In the paper we have developed the estimators of the parameters and we have obtained the corresponding asymptotic distribution for each estimator separately and for the joint estimators. An extensive simulation study has been carried out. This study seems to show that this new family of estimators are competitive in terms of efficiency
with MLE and other estimators appearing in the literature, whilst it has a better performance when contamination arises for values of the tuning parameter ranging from 0.2 to 0.5. As a conclusion, we feel that this family of estimators is an appealing option to estimate the parameters of the log-logistic distribution for this type of sampling.

As explained previously, we have chosen the DPD because it has been proved to lead to robust estimators. However, this is not the only family of divergence leading to robust estimators. Another option that could be considered is the family of R\'enyi's psuedo-distances. The development of these estimators and the comparison with MDPDE presented in this paper is a problem that we intend to treat in the future. In this sense, the problem of finding estimators based on minimizing R\'enyi's pseudodistance in the context of independent but not identically distributed data has been treated e.g. in (\cite{cajapa22, fejamipa24b}).

After introducing the MDPDE, it could be interesting to consider the restricted MDPDE in the line of (\cite{baghmapa22}). The restricted MDPDE was considered in Statistical Information Theory in (\cite{papazo02}) in relation to $\phi $-divergences or in (\cite{jamipa22}) for R\'enyi's pseudodistance. The restricted MDPDE is necessary to deal with Rao-type test statistics based on MDPDE (first considered in (\cite{baghmapa22})) for testing composite hypothesis for the parameters of the log-logistic distribution.

Similarly, it is also possible to consider Wald-type test statistics based on MDPDE for testing simple and composite null hypothesis for the parameters of the log-logistic distribution (see e.g. (\cite{bamamapa16, baghmamapa17, baghmapa18, bamamapa19, baghmamapa21, ghmamapa16, ghbapa21}) for Wald-type tests statistics based on MDPDE).

\section*{Acknowledgements}

This work was supported by the Spanish Grant PID2021-124933NB-I00.

\bibliographystyle{abbrvnat}

\appendix
\section{APPENDIX: Proofs of the results}

\subsection{Previous results}

In this subsection we will establish some results that will be used in the different proofs. Let us denote by $B\left( a,b\right) $ the beta function of
parameters $a$ and $b,$ $i.e,$
\begin{equation*}
	B\left( a,b\right) =\int_{0}^{1}x^{a-1}\left( 1-x\right) ^{b-1}dx.
\end{equation*}

We denote by $\Psi \left( x\right) $ the digamma function defined as the
logarithmic derivative of the gamma function. First, we need some previous results that will applied in the proofs of the results.

\begin{lemma}\label{lema1}(\cite{fejamipa24})
	Let $m(\beta )$ be a real function for which there exists the first
	derivative. We have,
	
	\begin{equation*}
		I_{1}=\int_{0}^{\infty }\frac{t^{m\left( \beta \right) }}{\left( 1+t\right)
			^{s}}dt=B\left( s-m\left( \beta \right) -1,m\left( \beta \right) +1\right) .
	\end{equation*}
\end{lemma}

\begin{corollary}\label{corI1}
	Let $m(\beta )$ be a real function for which there exists the first
	derivative. We have
	
	\begin{equation*}
		\int_0^{\infty } \frac{\left( \left( {x\over \alpha }\right)^{\beta }\right)^{m (\beta )}}{\left( 1+\left( {x\over \alpha }\right)^{\beta }\right)^{s}} dx= B\left( s-m\left( \beta \right) -1,m\left( \beta \right) +1\right) .
	\end{equation*}
\end{corollary}

\begin{proof}
	It suffices to make the change $t= \left( {x\over \alpha }\right)^{\beta }, dt=\frac{\beta }{\alpha }\left( {x\over \alpha } \right)^{\beta -1}dx$ so that
	
	\begin{equation*}
		dx=\frac{\alpha }{\beta } \left( \frac{x}{\alpha} \right)^{1-\beta }
		dt = \frac{\alpha }{\beta } t^{\frac{1-\beta }{\beta }}dt.
	\end{equation*}
	
	Then,
	
	\begin{eqnarray*}
		\int_0^{\infty } \frac{\left( \left( {x\over \alpha }\right)^{\beta }\right)^{m (\beta )}}{\left( 1+\left( {x\over \alpha }\right)^{\beta }\right)^{s}} dx & = & \int_0^{\infty } \frac{t^{m (\beta )}}{\left( 1+ t\right)^{s}} \frac{\alpha }{\beta } t^{\frac{1-\beta }{\beta }}dt \\
		& = & \frac{\alpha }{\beta } \int_0^{\infty } \frac{ t^{m (\beta ) + \frac{1-\beta }{\beta }}}{\left( 1+ t\right)^{s}} dt \\
		& = & \frac{\alpha }{\beta } B\left( s-m\left( \beta \right) - \frac{1-\beta }{\beta } -1,m\left( \beta \right) +{1\over \beta }\right) ,
	\end{eqnarray*}
	where we are applying Lemma \ref{lema1} for the last equality. This finishes the proof.
\end{proof}

\begin{lemma}(\cite{fejamipa24})
	\label{lema2}Let $m(\beta )$ be a real function for which there exists the first
	derivative. We have,
	\begin{eqnarray*}
		I_{2} &=&\int_{0}^{\infty }\left( \log t\right) \frac{t^{m\left( \beta
				\right) }}{\left( t+1\right) ^{s}}dt \\
		&=&B\left( s-m(\beta )-1,m(\beta )+1\right) \left \{ \Psi \left( m(\beta
		)+1\right) -\Psi \left( s-m(\beta )-1\right) \right \} .
	\end{eqnarray*}%
\end{lemma}

\begin{corollary}\label{corI2}
	Let $m(\beta )$ be a real function for which there exists the first
	derivative. We have
	
	\begin{equation*}
		\begin{aligned}
					\int_0^{\infty } \log \left( \left( {x\over \alpha }\right)^{\beta }\right) \frac{\left( \left( {x\over \alpha }\right)^{\beta }\right)^{m (\beta )}}{\left( 1+\left( {x\over \alpha }\right)^{\beta }\right)^{s}} dx= & B\left( s-m\left( \beta \right) -1,m\left( \beta \right) +1\right) \\
					& \times \left[ \Psi (m(\beta ) +1) - \Psi (s-m(\beta ) -1)\right] .
		\end{aligned}
	\end{equation*}
\end{corollary}

\begin{proof}
	As before, we make the change $t= \left( {x\over \alpha }\right)^{\beta },\quad dt=\frac{\beta }{\alpha }\left( {x\over \alpha } \right)^{\beta -1}dx,$ so that
	
	\begin{equation*}
		dx=\frac{\alpha }{\beta } \left( \frac{x}{\alpha} \right)^{1-\beta }
		dt = \frac{\alpha }{\beta } t^{\frac{1-\beta }{\beta }}dt.
	\end{equation*}
	
	Then,
	
	\begin{eqnarray*}
		\int_0^{\infty } \log \left( \left( {x\over \alpha }\right)^{\beta }\right) \frac{\left( \left( {x\over \alpha }\right)^{\beta }\right)^{m (\beta )}}{\left( 1+\left( {x\over \alpha }\right)^{\beta }\right)^{s}} dx & = & \int_0^{\infty } \log t \frac{t^{m (\beta )}}{\left( 1+ t\right)^{s}} \frac{\alpha }{\beta } t^{\frac{1-\beta }{\beta }}dt \\
		& = & \frac{\alpha }{\beta } \int_0^{\infty } \log t \frac{ t^{m (\beta ) + \frac{1-\beta }{\beta }}}{\left( 1+ t\right)^{s}} dt \\
		& = & \frac{\alpha }{\beta } B\left( s-m\left( \beta \right) - \frac{1-\beta }{\beta } -1,m\left( \beta \right) +{1\over \beta }\right) \\
		& & \left[ \Psi \left( m\left( \beta \right) +{1\over \beta }\right) - \Psi \left( s-m\left( \beta \right) - \frac{1-\beta }{\beta } -1\right) \right] ,
	\end{eqnarray*}
	where we are applying Lemma \ref{lema2} for the last equality. This finishes the proof.
\end{proof}

\begin{lemma}(\cite{fejamipa24})
	\label{lema3}Let $m(\beta )$ be a real function for which there exists the first
	derivative. For the log-logistic distribution we have,
	
	\begin{eqnarray*}
		I_{3} &=&\int_{0}^{\infty }\left( \log t\right) ^{2}\frac{t^{m\left( \beta
				\right) }}{\left( t+1\right) ^{s}}dt \\
		&=&B(s-m\left( \beta \right) -1,m\left( \beta \right) +1)\left \{ \left(
		\Psi \left( m\left( \beta \right) +1\right) -\Psi \left( s-m\left( \beta
		\right) -1\right) \right) ^{2}\right. \\
		&&\left. +\left( \Psi ^{\prime }\left( m\left( \beta \right) +1\right) +\Psi
		^{\prime }\left( s-m\left( \beta \right) -1\right) \right) \right \} .
	\end{eqnarray*}
\end{lemma}

\subsection{Proof of Theorem 1}

We have that

\begin{eqnarray*}
\int_0^{\infty } f_{\alpha , \beta }^{(i)}(x)^{\tau +1} dx & = & \int_0^{\infty } {c(i,n)^{\tau +1} \beta^{\tau +1}\over \alpha^{\tau +1}} {\left( {x\over \alpha }\right)^{(i\beta -1)(\tau +1)} \over \left( 1+ \left( {x\over \alpha }\right)^{\beta } \right)^{(n+1)(\tau +1)} } dx \\
& = & {c(i,n)^{\tau +1} \beta^{\tau +1}\over \alpha^{\tau +1}} \int_0^{\infty } {\left( {x\over \alpha }\right)^{(i\beta -1)(\tau +1)} \over \left( 1+ \left( {x\over \alpha }\right)^{\beta } \right)^{(n+1)(\tau +1)} } dx.
\end{eqnarray*}

We consider the change $t=\left( {x\over \alpha }\right)^{\beta },$ so that $dt= {\beta \over \alpha^{\beta }} x^{\beta -1} dx$ and hence $x^{\beta -1} dx = {\alpha^{\beta } \over \beta } dt.$ Therefore, as

$$ (i\beta -1)(\tau +1) - \beta +1 = i\beta \tau + i\beta -\tau -1 -\beta +1 = i\beta \tau + i\beta -\tau -\beta ,$$
it follows that

\begin{align*}
\int_0^{\infty } f_{\alpha , \beta }^{(i)}(x)^{\tau +1} dx   = & {c(i,n)^{\tau +1} \beta^{\tau +1}\over \alpha^{\tau +1}} {\alpha^{\beta -1}\over \beta } {1\over \alpha^{\beta -1}} \int_0^{\infty } {t^{{i\beta \tau + i\beta -\tau -\beta \over \beta }} \over (1+t)^{(n+1)(\tau +1)} } dt \\
 = &  c(i,n)^{\tau +1} \left( {\beta \over \alpha }\right)^{\tau } \\
 & \times B\left( (n+1)(\tau +1) - {i\beta \tau + i\beta -\tau -\beta \over \beta } -1, {i\beta \tau + i\beta -\tau -\beta \over \beta }+1 \right) .
\end{align*}

Finally,

\begin{eqnarray*}
(n+1)(\tau +1) - {i\beta \tau + i\beta -\tau -\beta \over \beta } -1 & = & {(n-i+1)\beta \tau + (n-i+1)\beta + \tau \over \beta },\\
{i\beta \tau + i\beta -\tau -\beta \over \beta }+1 & = & {i\beta \tau + i\beta + \tau \over \beta },
\end{eqnarray*}
and the result follows.

\subsection{Proof of Theorem 2}

We need to derive ${\partial H_{n,\tau }\left( \alpha ,\beta \right)\over \partial \alpha }$ and ${\partial H_{n,\tau }\left( \alpha ,\beta \right)\over \partial \beta }.$ Let us define

\begin{eqnarray*}
f_i(\alpha , \beta ) & = & c(i,n)^{\tau +1} \left( {\beta \over \alpha }\right)^{\tau } B\left( \frac{\left(
n+i-1\right) \tau \beta +\tau +\beta \left( n+i-1\right) }{\beta },\frac{i\beta \tau +i\beta -\tau }{\beta }\right) ,\\
g_i(\alpha , \beta ) & = & \left( 1 + {1\over \tau }\right) {c(i,n)^{\tau } \beta^{\tau } \left( {y_i \over \alpha }\right)^{\tau (i\beta -1)} \over \alpha^{\tau } \left( 1+ \left( {y_i \over \alpha }\right)^{\beta } \right)^{\tau (n+1)}} .
\end{eqnarray*}

Hence,

$$ H_{n,\tau }\left( \alpha ,\beta \right) ={1\over n} \sum_{i=1}^n f_i(\alpha , \beta ) + g_i(\alpha , \beta ).$$

Now,

\begin{eqnarray*}
{\partial f_i(\alpha , \beta )\over \partial \alpha } & = & c(i,n)^{\tau +1} \tau \left( {\beta \over \alpha }\right)^{\tau -1} {(-\beta )\over \alpha^2} B\left( \frac{\left(
n+i-1\right) \tau \beta +\tau +\beta \left( n+i-1\right) }{\beta },\frac{i\beta \tau +i\beta -\tau }{\beta }\right) \\
& = & c(i,n)^{\tau +1} \left( {\beta \over \alpha }\right)^{\tau } B\left( \frac{\left(
n+i-1\right) \tau \beta +\tau +\beta \left( n+i-1\right) }{\beta },\frac{i\beta \tau +i\beta -\tau }{\beta }\right) {(-\tau )\over \alpha}.
\end{eqnarray*}

For $g_i(\alpha , \beta )$, we have

\begin{eqnarray*}
g_i(\alpha , \beta ) & = & \left( 1 + {1\over \tau }\right) c(i,n)^{\tau } \beta^{\tau } y_i^{\tau (i\beta -1)} {\left( {1 \over \alpha }\right)^{\tau (i\beta -1)} \over \alpha^{\tau } \left( 1+ \left( {y_i \over \alpha }\right)^{\beta } \right)^{\tau (n+1)}} \\
& = & \left( 1 + {1\over \tau }\right) c(i,n)^{\tau } \beta^{\tau } y_i^{\tau (i\beta -1)} {\left( {1 \over \alpha }\right)^{\tau (i\beta -1)} \over \alpha^{\tau } \left( \alpha^{\beta }+ y_i^{\beta } \right)^{\tau (n+1)}{1\over \alpha^{\beta \tau (n+1)}}} \\
& = & \left( 1 + {1\over \tau }\right) c(i,n)^{\tau } \beta^{\tau } y_i^{\tau (i\beta -1)} {\alpha^{\beta \tau (n+1) - \tau (i\beta -1) - \tau } \over \left( \alpha^{\beta }+ y_i^{\beta } \right)^{\tau (n+1)}} \\
& = & \left( 1 + {1\over \tau }\right) c(i,n)^{\tau } \beta^{\tau } y_i^{\tau (i\beta -1)} {\alpha^{\beta \tau (n-i+1)} \over \left( \alpha^{\beta }+ y_i^{\beta } \right)^{\tau (n+1)}}.
\end{eqnarray*}

Therefore,

\begin{eqnarray*}
{\partial g_i(\alpha , \beta )\over \partial \alpha } & = &  \left( 1 + {1\over \tau }\right) c(i,n)^{\tau } \beta^{\tau } y_i^{\tau (i\beta -1)} \left[ {\beta \tau (n-i+1) \alpha^{\beta \tau (n-i+1)-1} \over \left( \alpha^{\beta }+ y_i^{\beta } \right)^{\tau (n+1)}} \right.    \\
& & \left. - {\alpha^{\beta \tau (n-i+1)} \left( \alpha^{\beta }+ y_i^{\beta } \right)^{\tau (n+1)-1} \tau (n+1)\beta \alpha^{\beta -1} \over \left( \alpha^{\beta }+ y_i^{\beta } \right)^{2\tau (n+1)}} \right] \\
& = & \left( 1 + {1\over \tau }\right) c(i,n)^{\tau } \beta^{\tau } y_i^{\tau (i\beta -1)} {\alpha^{\beta \tau (n-i+1)} \over \left( \alpha^{\beta }+ y_i^{\beta } \right)^{\tau (n+1)}} \left[ {\beta \tau (n-i+1)\over \alpha }- {\beta \tau (n+1) \alpha^{\beta -1} \over \left( \alpha^{\beta } + y_i^{\beta }\right) }\right]       \\
& = & \left( 1 + {1\over \tau }\right) c(i,n)^{\tau } \beta^{\tau } y_i^{\tau (i\beta -1)} {\alpha^{\beta \tau (n-i+1)} \over \left( \alpha^{\beta }+ y_i^{\beta } \right)^{\tau (n+1)}} \left[ {y_i^{\beta }\beta \tau (n-i+1) - \alpha^{\beta }\beta \tau i \over \alpha \left( \alpha^{\beta } + y_i^{\beta }\right) } \right] .
\end{eqnarray*}

Hence,

\begin{align*}
{\partial H_{n,\tau }\left( \alpha ,\beta \right)\over \partial \alpha } & = & {1\over n} \sum_{i=1}^n \left\{ c(i,n)^{\tau +1} \left( {\beta \over \alpha }\right)^{\tau } B\left( \frac{\left(
n+i-1\right) \tau \beta +\tau +\beta \left( n+i-1\right) }{\beta },\frac{i\beta \tau +i\beta -\tau }{\beta }\right) {(-\tau )\over \alpha} \right. \\
& & - \left. \left( 1 + {1\over \tau }\right) c(i,n)^{\tau } \beta^{\tau } y_i^{\tau (i\beta -1)} {\alpha^{\beta \tau (n-i+1)} \over \left( \alpha^{\beta }+ y_i^{\beta } \right)^{\tau (n+1)}} \left[ {y_i^{\beta }\beta \tau (n-i+1) - \alpha^{\beta }\beta \tau i \over \alpha \left( \alpha^{\beta } + y_i^{\beta }\right) } \right] \right\} .
\end{align*}

Deleting $c(i,n)^{\tau } \beta^{\tau }$ as a common factor, we obtain the first equation.

Next, to consider $ {\partial H_{n,\tau }\left( \alpha ,\beta \right)\over \partial \beta },$ we have to take into account that

$$ {\partial B\left( s-m(\beta )-1, m(\beta ) +1 \right) \over \partial \beta } = m'(\beta ) B\left( s-m(\beta )-1, m(\beta ) +1 \right) \left[ \Psi (m(\beta) +1) - \Psi (s-m(\beta ) - 1)\right] ,$$ where $\Psi $ denotes the digamma function. In this case, $m(\beta )= {i\beta \tau + i \beta - \beta - \tau \over \beta },$ and consequently

$$ m'(\beta )= {(i\tau  + i -1)\beta - (i\beta \tau + i \beta - \beta - \tau ) \over \beta^2 } = {\tau \over \beta^2}.$$

Applying this,

\begin{eqnarray*}
{\partial f_i(\alpha , \beta )\over \partial \beta } & = & c(i,n)^{\tau +1} \tau {\beta^{\tau -1} \over \alpha^{\tau } } B\left( \frac{\left(
n+i-1\right) \tau \beta +\tau +\beta \left( n+i-1\right) }{\beta },\frac{i\beta \tau +i\beta -\tau }{\beta }\right) \\
& & + c(i,n)^{\tau +1} \left( {\beta \over \alpha }\right)^{\tau } {\tau \over \beta^2} B\left( \frac{\left(
n+i-1\right) \tau \beta +\tau +\beta \left( n+i-1\right) }{\beta },\frac{i\beta \tau +i\beta -\tau }{\beta }\right) \\
& & \left[ \Psi\left( \frac{i\beta \tau +i\beta -\tau }{\beta } \right) - \Psi \left( \frac{\left(
n+i-1\right) \tau \beta +\tau +\beta \left( n+i-1\right) }{\beta }\right) \right] \\
& = & c(i,n)^{\tau +1} \tau \left( {\beta \over \alpha }\right)^{\tau -1} {1\over \alpha } B\left( \frac{\left(
n+i-1\right) \tau \beta +\tau +\beta \left( n+i-1\right) }{\beta },\frac{i\beta \tau +i\beta -\tau }{\beta }\right) \\
& & \left[ 1 + {1\over \beta }\left( \Psi\left( \frac{i\beta \tau +i\beta -\tau }{\beta } \right) - \Psi \left( \frac{\left(
n+i-1\right) \tau \beta +\tau +\beta \left( n+i-1\right) }{\beta }\right) \right) \right] .
\end{eqnarray*}

Finally,

\begin{eqnarray*}
{\partial g_i(\alpha , \beta )\over \partial \alpha } & = & { \left( 1 + {1\over \tau }\right) c(i,n)^{\tau }\over \alpha^{\tau }} \left\{ \left[ \tau \beta^{\tau -1} \left( {y_i\over \alpha }\right)^{\tau (i\beta -1)} + \beta^{\tau } \left( {y_i\over \alpha }\right)^{\tau (i\beta -1)} \tau i \ln \left( {y_i\over \alpha }\right) \right] \left( 1 + \left( {y_i\over \alpha }\right)^{\beta } \right) \right. \\
& & \left. - \beta^{\tau } \left( {y_i\over \alpha }\right)^{\tau (i\beta -1)} \tau (n+1) \left( {y_i\over \alpha }\right)^{\beta } \ln \left( {y_i\over \alpha }\right) \right\} \left( 1 + \left( {y_i\over \alpha }\right)^{\beta } \right)^{\tau (n -1) - 1}\\
& = & { \left( 1 + {1\over \tau }\right) c(i,n)^{\tau }\over \alpha^{\tau }} \tau \beta^{\tau } \left( {y_i\over \alpha }\right)^{\tau (i\beta -1)} \\
& & \left[ \left( {1\over \beta } + i \ln \left( {y_i\over \alpha }\right) \right) \left( 1 + \left( {y_i\over \alpha }\right)^{\beta } \right)^{\tau (n -1)} - (n+1) \left( {y_i\over \alpha }\right)^{\beta } \ln \left( {y_i\over \alpha }\right) \left( 1 + \left( {y_i\over \alpha }\right)^{\beta } \right)^{\tau (n -1)-1}
\right] .
\end{eqnarray*}

And we conclude

\begin{eqnarray*}
{\partial H_{n,\tau }\left( \alpha ,\beta \right)\over \partial \beta } & = & {1\over n} \sum_{i=1}^n c(i,n)^{\tau +1} \tau \left( {\beta \over \alpha }\right)^{\tau -1} {1\over \alpha } B\left( \frac{\left(
n+i-1\right) \tau \beta +\tau +\beta \left( n+i-1\right) }{\beta },\frac{i\beta \tau +i\beta -\tau }{\beta }\right) \\
& & \left[ 1 + {1\over \beta }\left( \Psi\left( \frac{i\beta \tau +i\beta -\tau }{\beta } \right) - \Psi \left( \frac{\left(
n+i-1\right) \tau \beta +\tau +\beta \left( n+i-1\right) }{\beta }\right) \right) \right] \\
& & - {\left( 1 + {1\over \tau }\right) c(i,n)^{\tau }\over \alpha^{\tau }} \tau \beta^{\tau } \left( {y_i\over \alpha }\right)^{\tau (i\beta -1)} \\
& & \left[ \left( {1\over \beta } + i \ln \left( {y_i\over \alpha }\right) \right) \left( 1 + \left( {y_i\over \alpha }\right)^{\beta } \right)^{\tau (n -1)} - (n+1) \left( {y_i\over \alpha }\right)^{\beta } \ln \left( {y_i\over \alpha }\right) \left( 1 + \left( {y_i\over \alpha }\right)^{\beta } \right)^{\tau (n -1)-1}
\right] .
\end{eqnarray*}

Removing the common factor ${c(i,n)^{\tau } \beta^{\tau -1} \over \alpha^{\tau }}$ the second equation arises.

\subsection{Proof of Theorem 3}

As

$$ f_{\alpha }^{(i)}(x) = {c(i,n) \beta x^{i\beta -1} \alpha^{\beta (n-i+1)} \over (\alpha^{\beta } + x^{\beta })^{n+1}},$$
it follows that

\begin{equation*}
\frac{\partial \log f_{\alpha }^{(i)}(x)}{\partial \alpha }= \frac{\beta (n-i+1)}{\alpha }%
-\frac{(n+1)}{\left( \alpha^{\beta } + x^{\beta }\right)}\cdot \beta \alpha^{\beta -1} .
\end{equation*}

Therefore,

\begin{equation*}
\left( \frac{\partial \log f_{\alpha }^{(i)}(x)}{\partial \alpha }\right)^2 = \left({\beta \over \alpha }\right)^2 (n-i+1)^2 + {(n+1)^2 \beta^2 \alpha^{2\beta -1} \over \left( \alpha^{\beta } + x^{\beta }\right)^2} -2 (n-i+1) (n+1) \beta^2 \alpha^{\beta -2} {1\over \alpha^{\beta } + x^{\beta }}.
\end{equation*}

Consequently,

\begin{eqnarray*}
J_{\tau }^{(i)}\left( \alpha \right) & = & \int_{0}^{\infty } \left[ \left({\beta \over \alpha }\right)^2 (n-i+1)^2 + {(n+1)^2 \beta^2 \alpha^{2\beta -1} \over \left( \alpha^{\beta } + x^{\beta }\right)^2} -2 (n-i+1) (n+1) \beta^2 \alpha^{\beta -2} {1\over \alpha^{\beta } + x^{\beta }} \right] \\
& & \cdot {c(i,n)^{\tau +1} \beta^{\tau +1} x^{(i\beta -1)(\tau +1)} \alpha^{\beta (\tau +1)(n-i+1)} \over (\alpha^{\beta } + x^{\beta })^{(n+1)(\tau +1)}} dx \\
& = & A_1 + A_2 -A_3.
\end{eqnarray*}

Let us find each of these parts separately.

\begin{eqnarray*}
A_1 & = & \left({\beta \over \alpha }\right)^2 (n-i+1)^2 \int_{0}^{\infty } {c(i,n)^{\tau +1} \beta^{\tau +1} x^{(i\beta -1)(\tau +1)} \alpha^{\beta (\tau +1)(n-i+1)} \over (\alpha^{\beta } + x^{\beta })^{(n+1)(\tau +1)}} dx \\
& = & \left({\beta \over \alpha }\right)^2 (n-i+1)^2 c(i,n)^{\tau +1} \beta^{\tau +1} \alpha^{\beta (\tau +1)(n-i+1)} \cdot \int_{0}^{\infty } {\left( {x\over \alpha } \right)^{(i\beta -1)(\tau +1)} \alpha^{(i\beta -1)(\tau +1)} \over \left( 1 + \left({x\over \alpha }\right)^{\beta }\right)^{(n+1)(\tau +1)} \alpha^{\beta (n+1)(\tau +1)}} dx \\
& = & \left({\beta \over \alpha }\right)^2 (n-i+1)^2 c(i,n)^{\tau +1} \beta^{\tau +1} \alpha^{(i\beta -1)(\tau +1)} \alpha^{-\beta i(\tau +1)} \cdot \int_{0}^{\infty } {\left( \left( {x\over \alpha } \right)^{\beta } \right)^{i\beta \tau + i\beta -\beta -\tau \over \beta} \left( {x\over \alpha }\right)^{\beta -1} \over \left( 1 + \left({x\over \alpha }\right)^{\beta }\right)^{(n+1)(\tau +1)}} dx \\
& = & \left({\beta \over \alpha }\right)^2 (n-i+1)^2 c(i,n)^{\tau +1} \beta^{\tau +1} \alpha^{(i\beta -1)(\tau +1)} \alpha^{-\beta i(\tau +1)} {\alpha \over \beta } \int_0^1 {t^{i\beta \tau + i\beta -\beta -\tau \over \beta} \over (1+t)^{(n+1)(\tau +1)}} dx \\
& = & \left({\beta \over \alpha }\right)^2 (n-i+1)^2 c(i,n)^{\tau +1} \left( {\beta \over \alpha }\right)^{\tau } B\left( {(n-i+1) \tau \beta + (n-i+1) \beta +\tau \over \beta }, {i \tau \beta + i \beta -\tau \over \beta }\right) \\
& = & \left({\beta \over \alpha }\right)^{\tau +2} (n-i+1)^2 c(i,n)^{\tau +1} B\left( {(n-i+1) \tau \beta + (n-i+1) \beta +\tau \over \beta }, {i \tau \beta + i \beta -\tau \over \beta }\right) .
\end{eqnarray*}

\begin{eqnarray*}
A_2 & = & \beta^2 \alpha^{2\beta -2} (n-i+1)^2 \int_{0}^{\infty } {c(i,n)^{\tau +1} \beta^{\tau +1} x^{(i\beta -1)(\tau +1)} \alpha^{\beta (\tau +1)(n-i+1)} \over (\alpha^{\beta } + x^{\beta })^{(n+1)(\tau +1)+2}} dx \\
& = & \beta^2 \alpha^{2\beta -2} (n-i+1)^2 c(i,n)^{\tau +1} \beta^{\tau +1} \alpha^{\beta (\tau +1)(n-i+1)} \cdot \int_{0}^{\infty } { x^{\beta -1} x^{i\beta \tau + i\beta -\beta -\tau} \over (\alpha^{\beta } + x^{\beta })^{(n+1)(\tau +1)+2}} dx \\
& = & \beta^2 \alpha^{2\beta -2} (n-i+1)^2 c(i,n)^{\tau +1} \beta^{\tau +1} \alpha^{\beta (\tau +1)(n-i+1)} \alpha^{i\beta \tau + i\beta -\beta -\tau} \alpha^{-\beta \left( (n+1)(\tau +1)+2\right) } \cdot \\
& & \int_{0}^{\infty } {\left( \left( {x\over \alpha } \right)^{\beta } \right)^{i\beta \tau + i\beta -\beta -\tau \over \beta} x^{\beta -1} \over \left( 1 + \left({x\over \alpha }\right)^{\beta }\right)^{(n+1)(\tau +1)+2}} dx \\
& = & \beta^2 \alpha^{2\beta -2} (n-i+1)^2 c(i,n)^{\tau +1} \beta^{\tau +1} \alpha^{\beta (\tau +1)(n-i+1)} \alpha^{i\beta \tau + i\beta -\beta -\tau} \alpha^{-\beta \left( (n+1)(\tau +1)+2\right) } \alpha^{\beta } \beta^{-1} \cdot \\
& &  B\left( {(n-i+1) \tau \beta + (n-i+3) \beta +\tau \over \beta }, {i \tau \beta + i \beta -\tau \over \beta }\right) \\
& = & \left({\beta \over \alpha }\right)^{\tau +2} (n-i+1)^2 c(i,n)^{\tau +1} B\left( {(n-i+1) \tau \beta + (n-i+3) \beta +\tau \over \beta }, {i \tau \beta + i \beta -\tau \over \beta }\right) .
\end{eqnarray*}

Now,

\begin{equation*}
	\begin{aligned}
		 & B\left( {(n-i+1) \tau \beta + (n-i+3) \beta +\tau \over \beta }, {i \tau \beta + i \beta -\tau \over \beta }\right) \\
		 & = {(n-i+1) \tau \beta + (n-i+2) \beta +\tau \over (n+1) \tau \beta + (n+2) \beta } \cdot B\left( {(n-i+1) \tau \beta + (n-i+2) \beta +\tau \over \beta }, {i \tau \beta + i \beta -\tau \over \beta }\right)\\
		 & = {(n-i+1) \tau \beta + (n-i+2) \beta +\tau \over (n+1) \tau \beta + (n+2) \beta } \cdot {(n-i+1) \tau \beta + (n-i+1) \beta +\tau \over (n+1) \tau \beta + (n+1) \beta } \\
		 & \hspace{1cm} \times B\left( {(n-i+1) \tau \beta + (n-i+1) \beta +\tau \over \beta }, {i \tau \beta + i \beta -\tau \over \beta }\right) 
	\end{aligned}
\end{equation*}

Consequently,

\begin{align*}
A_2  = & \left({\beta \over \alpha }\right)^{\tau +2} (n-i+1)^2 c(i,n)^{\tau +1} \cdot {(n-i+1) \tau \beta + (n-i+2) \beta +\tau \over \beta \left[ (n+1) \tau + (n+2) \right] }\\
& \times {(n-i+1) \tau \beta + (n-i+1) \beta +\tau \over \beta \left[ (n+1) \tau + (n+1) \right] }  \cdot B\left( {(n-i+1) \tau \beta + (n-i+1) \beta +\tau \over \beta }, {i \tau \beta + i \beta -\tau \over \beta }\right) .
\end{align*}

\begin{eqnarray*}
A_3 & = & 2\left( n -i +1\right) \left( n+1\right) \beta^2 \alpha^{\beta -2} \int_{0}^{\infty } {c(i,n)^{\tau +1} \beta^{\tau +1} x^{(i\beta -1)(\tau +1)} \alpha^{\beta (\tau +1)(n-i+1)} \over (\alpha^{\beta } + x^{\beta })^{(n+1)(\tau +1)+1}} dx \\
& = & 2\left( n -i +1\right) \left( n+1\right) \beta^2 \alpha^{\beta -2} c(i,n)^{\tau +1} \beta^{\tau +1} \alpha^{\beta (\tau +1)(n-i+1)} \cdot \int_{0}^{\infty } { x^{\beta -1} x^{i\beta \tau + i\beta -\beta -\tau} \over (\alpha^{\beta } + x^{\beta })^{(n+1)(\tau +1)+1}} dx \\
& = & 2\left( n -i +1\right) \left( n+1\right) \beta^2 \alpha^{\beta -2} c(i,n)^{\tau +1} \beta^{\tau +1} \alpha^{\beta (\tau +1)(n-i+1)} \alpha^{i\beta \tau + i\beta -\beta -\tau} \alpha^{-\beta \left( (n+1)(\tau +1)+1\right) } \\
& & \int_{0}^{\infty } {\left( \left( {x\over \alpha } \right)^{\beta } \right)^{i\beta \tau + i\beta -\beta -\tau \over \beta} x^{\beta -1} \over \left( 1 + \left({x\over \alpha }\right)^{\beta }\right)^{(n+1)(\tau +1)+1}} dx \\
& = & 2\left( n -i +1\right) \left( n+1\right) \beta^2 \alpha^{\beta -2} c(i,n)^{\tau +1} \beta^{\tau +1} \alpha^{\beta (\tau +1)(n-i+1)} \alpha^{i\beta \tau + i\beta -\beta -\tau} \alpha^{-\beta \left( (n+1)(\tau +1)+1\right) } \alpha^{\beta } \beta^{-1} \cdot \\
& &  B\left( {(n-i+1) \tau \beta + (n-i+2) \beta +\tau \over \beta }, {i \tau \beta + i \beta -\tau \over \beta }\right) \\
& = & \left({\beta \over \alpha }\right)^{\tau +2} 2(n-i+1) (n+1) c(i,n)^{\tau +1} B\left( {(n-i+1) \tau \beta + (n-i+2) \beta +\tau \over \beta }, {i \tau \beta + i \beta -\tau \over \beta }\right) .
\end{eqnarray*}

As

$$ B\left( {(n-i+1) \tau \beta + (n-i+2) \beta +\tau \over \beta }, {i \tau \beta + i \beta -\tau \over \beta }\right) $$

$$ = {(n-i+1) \tau \beta + (n-i+1) \beta +\tau \over (n+1) \tau \beta + (n+1) \beta } \cdot B\left( {(n-i+1) \tau \beta + (n-i+1) \beta +\tau \over \beta }, {i \tau \beta + i \beta -\tau \over \beta }\right) ,$$
we conclude that

\begin{eqnarray*}
A_3 & = & \left({\beta \over \alpha }\right)^{\tau +2} 2(n-i+1) (n+1) c(i,n)^{\tau +1} \cdot {(n-i+1) \tau \beta + (n-i+1) \beta +\tau \over (n+1) \tau \beta + (n+1) \beta } \cdot \\
& & B\left( {(n-i+1) \tau \beta + (n-i+1) \beta +\tau \over \beta }, {i \tau \beta + i \beta -\tau \over \beta }\right) .
\end{eqnarray*}

\subsection{Proof of Corollary 4}

It suffices to take $\tau =0$ in Theorem \ref{Theorem1} for $A_1, A_2, A_3.$ Then,

\begin{eqnarray*}
A_1 & = & \left({\beta \over \alpha }\right)^{2} (n-i+1)^2 c(i,n)^{1} B\left( n-i+1, i \right) ,\\
A_2 & = & \left({\beta \over \alpha }\right)^{2} (n-i+1)^2 c(i,n)^{1} \cdot {n-i+2 \over n+2} \cdot {n-i+1 \over n+1} B\left( n-i+1, i \right) , \\
A_3 & = & \left({\beta \over \alpha }\right)^{2} 2(n-i+1) (n+1) c(i,n)^{1} \cdot {(n-i+1) \over n+1} \cdot B\left( n-i+1, i \right).
\end{eqnarray*}

Hence,

\begin{eqnarray*}
A_1 + A_2 - A_3 & = & \left( {\beta \over \alpha }\right)^2 c(i,n) B(n-i+1, i) \\
& & \left[ (n-i+1)^2 + {(n+1)(n-i+2)(n-i+1) \over n+2} -2 (n-i+1)^2 \right] \\
& = & \left( {\beta \over \alpha }\right)^2 c(i,n) B(n-i+1, i) (n-i+1) \left[ {(n+1)(n-i+2)(n-i+1) \over n+2} -(n-i+1)^2 \right] \\
& = & \left( {\beta \over \alpha }\right)^2 c(i,n) B(n-i+1, i) (n-i+1) {i\over n+2}.
\end{eqnarray*}

\subsection{Proof of Theorem 5}

We already know that

$$ \frac{\partial \log f_{\alpha }^{(i)}(x)}{\partial \alpha } = {\beta (n-i+1)\over \alpha } - {(n+1) \beta \alpha^{\beta -1} \over \alpha^{\beta } + x^{\beta }}.$$

Consequently,

\begin{eqnarray*}
\xi_{\tau }^{(i)}(\alpha ) & = & \int_0^{\infty } \left( {\beta (n-i+1)\over \alpha } - {(n+1) \beta \alpha^{\beta -1} \over \alpha^{\beta } + x^{\beta }}\right) \cdot {c(i,n)^{\tau +1} \beta^{\tau +1} x^{(i\beta -1)(\tau +1)} \alpha^{\beta (\tau +1)(n-i+1)} \over (\alpha^{\beta } + x^{\beta })^{(n+1)(\tau +1)}} dx \\
& = & B_1- B_2
\end{eqnarray*}

Now,

\begin{eqnarray*}
B_1 & = & \int_0^{\infty } {\beta (n-i+1)\over \alpha } \cdot {c(i,n)^{\tau +1} \beta^{\tau +1} x^{(i\beta -1)(\tau +1)} \alpha^{\beta (\tau +1)(n-i+1)} \over (\alpha^{\beta } + x^{\beta })^{(n+1)(\tau +1)}} dx \\
& = &  {\beta (n-i+1)\over \alpha } \cdot \int_0^{\infty }{c(i,n)^{\tau +1} \beta^{\tau +1} x^{(i\beta -1)(\tau +1)} \alpha^{\beta (\tau +1)(n-i+1)} \over (\alpha^{\beta } + x^{\beta })^{(n+1)(\tau +1)}} dx \\
& = & {\beta (n-i+1)\over \alpha } \left( {\beta \over \alpha }\right)^{\tau } c(i,n)^{\tau +1} B\left( {(n-i+1) \beta \tau + (n-i+1) \beta + \tau \over \beta }, {i \beta \tau + i \beta - \tau \over \beta }\right) \\
& = & \left( {\beta \over \alpha }\right)^{\tau +1} (n-i+1) c(i,n)^{\tau +1} B\left( {(n-i+1) \beta \tau + (n-i+1) \beta + \tau \over \beta }, {i \beta \tau + i \beta - \tau \over \beta }\right) ,
\end{eqnarray*}
where we are using that the integral has been already computed for $A_1$. On the other hand,

\begin{eqnarray*}
B_2 & = & \int_0^{\infty } {(n+1) \beta \alpha^{\beta -1} \over \alpha^{\beta } + x^{\beta }} \cdot {c(i,n)^{\tau +1} \beta^{\tau +1} x^{(i\beta -1)(\tau +1)} \alpha^{\beta (\tau +1)(n-i+1)} \over (\alpha^{\beta } + x^{\beta })^{(n+1)(\tau +1)}} dx \\
& = & (n+1) \beta \alpha^{\beta -1} \int_0^{\infty } {c(i,n)^{\tau +1} \beta^{\tau +1} x^{(i\beta -1)(\tau +1)} \alpha^{\beta (\tau +1)(n-i+1)} \over (\alpha^{\beta } + x^{\beta })^{(n+1)(\tau +1)+1}} dx \\
& = & (n+1) \beta \alpha^{\beta -1} c(i,n)^{\tau +1} \beta^{\tau } \alpha^{-\tau -\beta } B\left( {(n-i+1) \beta \tau + (n-i+2) \beta + \tau \over \beta }, {i \beta \tau + i \beta - \tau \over \beta }\right) \\
& = & \left( {\beta \over \alpha }\right)^{\tau +1} c(i,n)^{\tau +1} (n+1) {(n-i+1) \beta \tau + (n-i+1) \beta + \tau \over \beta (n+1)(\tau +1)} \\
& & \hspace{1cm} \times B\left( {(n-i+1) \beta \tau + (n-i+1) \beta + \tau \over \beta }, {i \beta \tau + i \beta - \tau \over \beta }\right) \\
& = & \left( {\beta \over \alpha }\right)^{\tau +1} c(i,n)^{\tau +1}  B\left( {(n-i+1) \beta \tau + (n-i+1) \beta + \tau \over \beta }, {i \beta \tau + i \beta - \tau \over \beta }\right)\\
& & \hspace{1cm} \times {(n-i+1) \beta \tau + (n-i+1) \beta + \tau \over \beta (\tau +1)}
\end{eqnarray*}
where we are using that the integral has been already computed for $A_3$. Consequently,

\begin{eqnarray*}
\xi_{\tau }^{(i)}(\beta ) & = & \left( {\beta \over \alpha }\right)^{\tau +1} c(i,n)^{\tau +1} B\left( {(n-i+1) \beta \tau + (n-i+1) \beta + \tau \over \beta }, {i \beta \tau + i \beta - \tau \over \beta }\right) \\
& & \left[ (n-i+1) -{(n-i+1) \beta \tau + (n-i+1) \beta + \tau \over \beta (\tau +1)} \right] \\
& = & \left( {\beta \over \alpha }\right)^{\tau +1} c(i,n)^{\tau +1} B\left( {(n-i+1) \beta \tau + (n-i+1) \beta + \tau \over \beta }, {i \beta \tau + i \beta - \tau \over \beta }\right) \cdot{-\tau \over \beta \tau + \beta } .
\end{eqnarray*}

\subsection{Proof of Theorem 7}

As

$$ f_{\beta }^{(i)} = {c(i,n) \beta x^{i\beta -1} \alpha^{\beta (n-i+1)} \over (\alpha^{\beta } + x^{\beta })^{n+1}},$$
it follows that

\begin{equation*}
\frac{\partial \log f_{\beta }^{(i)}(x)}{\partial \beta }= {1\over \beta } + i\log \left( {x\over \alpha }\right) - {(n+1) \over \left( 1 + \left( {x\over \alpha }\right)^{\beta }\right)} \left( {x\over \alpha }\right)^{\beta } \log \left( {x\over \alpha }\right).
\end{equation*}

Therefore,

\begin{align*}
\left( \frac{\partial \log f_{\beta }^{(i)}(x)}{\partial \beta }\right)^2 & =  {1\over \beta^2} + i^2 \left( \log \left( {x\over \alpha }\right) \right)^2 + {(n+1)^2 \over \left( 1 + \left( {x\over \alpha }\right)^{\beta }\right)^2} \left( \left( {x\over \alpha }\right)^{\beta }\right)^2 \left( \log \left( {x\over \alpha }\right) \right)^2 \\
&  + {2i \over \beta }\log \left( {x\over \alpha }\right) - {2(n+1) \over \beta }  \log \left( {x\over \alpha }\right) {\left( {x\over \alpha }\right)^{\beta }\over \left( 1 + \left( {x\over \alpha }\right)^{\beta }\right)}  - 2i (n+1) \left( \log \left( {x\over \alpha }\right) \right)^2 {\left( {x\over \alpha }\right)^{\beta }\over \left( 1 + \left( {x\over \alpha }\right)^{\beta }\right)}.
\end{align*}

Consequently,

\begin{eqnarray*}
J_{\tau }^{(i)}\left( \beta \right) & = & \int_{0}^{\infty } \left[ {1\over \beta^2} + i^2 \left( \log \left( {x\over \alpha }\right) \right)^2 + {(n+1)^2 \over \left( 1 + \left( {x\over \alpha }\right)^{\beta }\right)^2} \left( \left( {x\over \alpha }\right)^{\beta }\right)^2 \left( \log \left( {x\over \alpha }\right) \right)^2 \right. \\
& & \left. + {2i \over \beta }\log \left( {x\over \alpha }\right) - {2(n+1) \over \beta }  \log \left( {x\over \alpha }\right) {\left( {x\over \alpha }\right)^{\beta }\over \left( 1 + \left( {x\over \alpha }\right)^{\beta }\right)}  - 2i (n+1) \left( \log \left( {x\over \alpha }\right) \right)^2 {\left( {x\over \alpha }\right)^{\beta }\over \left( 1 + \left( {x\over \alpha }\right)^{\beta }\right)} \right] \\
& & \cdot {c(i,n)^{\tau +1} \beta^{\tau +1} x^{(i\beta -1)(\tau +1)} \alpha^{\beta (\tau +1)(n-i+1)} \over (\alpha^{\beta } + x^{\beta })^{(n+1)(\tau +1)}} dx \\
& = & C_1 + C_2 - C_3 + C_4- C_5 - C_6.
\end{eqnarray*}

Let us find each of these parts separately.

\begin{eqnarray*}
C_1 & = & \int_{0}^{\infty } {1\over \beta^2 } {c(i,n)^{\tau +1} \beta^{\tau +1} x^{(i\beta -1)(\tau +1)} \alpha^{\beta (\tau +1)(n-i+1)} \over (\alpha^{\beta } + x^{\beta })^{(n+1)(\tau +1)}} dx \\
& = & \int_{0}^{\infty } {1\over \beta^2 } \cdot {c(i,n)^{\tau +1} \beta^{\tau +1} \over \alpha^{\tau +1}} \cdot { \left( \left( {x\over \alpha }\right)^{\beta } \right)^{{i\beta \tau + i\beta -\beta - \tau \over \beta}} x^{\beta -1} {1\over \alpha^{\beta -1}}\over \left( 1 + \left( {x\over \alpha }\right)^{\beta } \right)^{(n+1)(\tau + 1)} } dx \\
& = &  \int_{0}^{\infty } {1\over \beta^2 } \cdot {c(i,n)^{\tau +1} \beta^{\tau +1} \over \alpha^{\tau +1}} \cdot { t^{{i\beta \tau + i\beta -\beta - \tau \over \beta}} {\alpha^{\beta }\over \beta } {1\over \alpha^{\beta -1}}\over \left( 1 + t\right)^{(n+1)(\tau + 1)} } dt \\
& = & {c(i,n)^{\tau +1}\over \beta^2 } \cdot { \beta^{\tau } \over \alpha^{\tau }} B\left( {(n+1-i)\tau \beta + (n+1-i)\beta + \tau \over \beta }, {i\beta \tau + i\beta - \tau \over \beta }\right) .
\end{eqnarray*}

\begin{align*}
C_2  = & \int_{0}^{\infty } i^2 \left( \log \left( {x\over \alpha }\right) \right)^2 \cdot {c(i,n)^{\tau +1} \beta^{\tau +1} x^{(i\beta -1)(\tau +1)} \alpha^{\beta (\tau +1)(n-i+1)} \over (\alpha^{\beta } + x^{\beta })^{(n+1)(\tau +1)}} dx \\
 = & \int_{0}^{\infty } {i^2\over \beta^2} \left( \log \left( {x\over \alpha }\right)^{\beta } \right)^2 \cdot {c(i,n)^{\tau +1} \beta^{\tau +1} \over \alpha^{\tau +1}} \cdot { \left( \left( {x\over \alpha }\right)^{\beta } \right)^{{i\beta \tau + i\beta -\beta - \tau \over \beta}} x^{\beta -1} {1\over \alpha^{\beta -1}}\over \left( 1 + \left( {x\over \alpha }\right)^{\beta } \right)^{(n+1)(\tau + 1)} } dx \\
 = & {i^2\over \beta^2} \cdot {c(i,n)^{\tau +1} \beta^{\tau +1} \over \alpha^{\tau +1}} \cdot {\alpha \over \beta } \int_0^{\infty }\left(\log t\right)^2 \cdot { t^{{i\beta \tau + i\beta -\beta - \tau \over \beta}} \over \left( 1 + t\right)^{(n+1)(\tau + 1)} } dt \\
 = & {i^2c(i,n)^{\tau +1}\over \beta^2} \cdot {\beta^{\tau } \over \alpha^{\tau }} \cdot B\left( {(n+1-i)\tau \beta + (n+1-i)\beta + \tau \over \beta }, {i\beta \tau + i\beta - \tau \over \beta }\right)\\
 & \times \left[ \Psi' \left( {i\beta \tau + i\beta - \tau \over \beta }\right)  + \Psi' \left( {(n+1-i)\tau \beta + (n+1-i)\beta + \tau \over \beta }\right)\right. \\
 & \left. \times \left( \Psi \left( {i\beta \tau + i\beta - \tau \over \beta }\right) - \Psi \left( {(n+1-i)\tau \beta + (n+1-i)\beta + \tau \over \beta }\right) \right)^2 \right] .
\end{align*}

\begin{align*}
C_3 = & \int_{0}^{\infty } {(n+1)^2 \over \left( 1 + \left( {x\over \alpha }\right)^{\beta }\right)^2} \left( \left( {x\over \alpha }\right)^{\beta }\right)^2 \left( \log \left( {x\over \alpha }\right) \right)^2  \cdot {c(i,n)^{\tau +1} \beta^{\tau +1} x^{(i\beta -1)(\tau +1)} \alpha^{\beta (\tau +1)(n-i+1)} \over (\alpha^{\beta } + x^{\beta })^{(n+1)(\tau +1)}} dx \\
 = & \int_{0}^{\infty } {(n+1)^2 \over \beta^2} \cdot {c(i,n)^{\tau +1} \beta^{\tau +1}\over \alpha^{\tau +1}} \cdot \left( \log \left( {x\over \alpha }\right)^{\beta } \right)^2 \cdot { \left( \left( {x\over \alpha }\right)^{\beta } \right)^{{i\beta \tau + i\beta +\beta - \tau \over \beta}} x^{\beta -1} {1\over \alpha^{\beta -1}}\over \left( 1 + \left( {x\over \alpha }\right)^{\beta } \right)^{(n+1)(\tau + 1)+2} } dx \\
 = & {(n+1)^2 \over \beta^2} \cdot {c(i,n)^{\tau +1} \beta^{\tau +1}\over \alpha^{\tau +1}} \cdot {\alpha \over \beta } \int_0^{\infty } \left( \log t\right)^2 \cdot { t^{{i\beta \tau + i\beta +\beta - \tau \over \beta}} \over \left( 1 + t\right)^{(n+1)(\tau + 1)+2} } dt \\
 = & {(n+1)^2 \over \beta^2} \cdot {c(i,n)^{\tau +1} \beta^{\tau }\over \alpha^{\tau }} \cdot B\left( {(n+1-i)\tau \beta + (n+1-i)\beta + \tau \over \beta }, {i\beta \tau + (i+2)\beta - \tau \over \beta }\right) \\
 &  \left[ \Psi' \left( {i\beta \tau + (i+2)\beta - \tau \over \beta }\right)  + \Psi' \left( {(n+1-i)\tau \beta + (n+1-i)\beta + \tau \over \beta }\right) \right. \\
 & \left. \times \left( \Psi \left( {i\beta \tau + (i+2)\beta - \tau \over \beta }\right) - \Psi \left( {(n+1-i)\tau \beta + (n+1-i)\beta + \tau \over \beta }\right) \right)^2 \right] .
\end{align*}

\begin{eqnarray*}
C_4 & = & \int_{0}^{\infty } {2i \over \beta }\log \left( {x\over \alpha }\right) \cdot {c(i,n)^{\tau +1} \beta^{\tau +1} x^{(i\beta -1)(\tau +1)} \alpha^{\beta (\tau +1)(n-i+1)} \over (\alpha^{\beta } + x^{\beta })^{(n+1)(\tau +1)}} dx \\
& = & \int_{0}^{\infty } {2i \over \beta^2 }\log \left( {x\over \alpha }\right)^{\beta } \cdot {c(i,n)^{\tau +1} \beta^{\tau +1}\over \alpha^{\tau +1}} \cdot { \left( \left( {x\over \alpha }\right)^{\beta } \right)^{{i\beta \tau + i\beta -\beta - \tau \over \beta}} x^{\beta -1} {1\over \alpha^{\beta -1}}\over \left( 1 + \left( {x\over \alpha }\right)^{\beta } \right)^{(n+1)(\tau + 1)}} dx \\
& = & {2i \over \beta^2 } \cdot {c(i,n)^{\tau +1} \beta^{\tau +1}\over \alpha^{\tau +1}} {\alpha \over \beta } \cdot \int_0^{infty } \log t \cdot { t^{{i\beta \tau + i\beta -\beta - \tau \over \beta}} \over \left( 1 + t\right)^{(n+1)(\tau + 1)} } dt \\
& = &  {2i \over \beta^2 } \cdot {c(i,n)^{\tau +1} \beta^{\tau }\over \alpha^{\tau }} \cdot B\left( {(n+1-i)\tau \beta + (n+1-i)\beta + \tau \over \beta }, {i\beta \tau + i\beta - \tau \over \beta }\right) \\
& & \cdot \left[ \Psi \left( {i\beta \tau + i\beta - \tau \over \beta }\right) - \Psi \left( {(n+1-i)\tau \beta + (n+1-i)\beta + \tau \over \beta } \right) \right] .
\end{eqnarray*}

\begin{eqnarray*}
C_5 & = & \int_{0}^{\infty } {2(n+1) \over \beta }  \log \left( {x\over \alpha }\right) {\left( {x\over \alpha }\right)^{\beta }\over \left( 1 + \left( {x\over \alpha }\right)^{\beta }\right)} \cdot {c(i,n)^{\tau +1} \beta^{\tau +1} x^{(i\beta -1)(\tau +1)} \alpha^{\beta (\tau +1)(n-i+1)} \over (\alpha^{\beta } + x^{\beta })^{(n+1)(\tau +1)}} dx \\
& = & \int_{0}^{\infty } {2(n+1) \over \beta^2 }\log \left( {x\over \alpha }\right)^{\beta } \cdot {c(i,n)^{\tau +1} \beta^{\tau +1}\over \alpha^{\tau +1}} \cdot { \left( \left( {x\over \alpha }\right)^{\beta } \right)^{{i\beta \tau + i\beta - \tau \over \beta}} x^{\beta -1} {1\over \alpha^{\beta -1}}\over \left( 1 + \left( {x\over \alpha }\right)^{\beta } \right)^{(n+1)(\tau + 1)+1}} dx \\
& = & {2(n+1) \over \beta^2 } \cdot {c(i,n)^{\tau +1} \beta^{\tau +1}\over \alpha^{\tau +1}} {\alpha \over \beta } \cdot \int_0^{\infty } \log t \cdot { t^{{i\beta \tau + i\beta - \tau \over \beta}} \over \left( 1 + t\right)^{(n+1)(\tau + 1)+1} } dt \\
& = &  {2(n+1) \over \beta^2 } \cdot {c(i,n)^{\tau +1} \beta^{\tau }\over \alpha^{\tau }} \cdot B\left( {(n+1-i)\tau \beta + (n+1-i)\beta + \tau \over \beta }, {i\beta \tau + (i+1)\beta - \tau \over \beta }\right) \\
& & \cdot \left[ \Psi \left( {i\beta \tau + (i+1)\beta - \tau \over \beta }\right) - \Psi \left( {(n+1-i)\tau \beta + (n+1-i)\beta + \tau \over \beta } \right) \right] .
\end{eqnarray*}

\begin{align*}
C_6  = & \int_{0}^{\infty } 2i (n+1) \left( \log \left( {x\over \alpha }\right) \right)^2 {\left( {x\over \alpha }\right)^{\beta }\over \left( 1 + \left( {x\over \alpha }\right)^{\beta }\right)} \cdot {c(i,n)^{\tau +1} \beta^{\tau +1} x^{(i\beta -1)(\tau +1)} \alpha^{\beta (\tau +1)(n-i+1)} \over (\alpha^{\beta } + x^{\beta })^{(n+1)(\tau +1)}} dx \\
= & \int_{0}^{\infty } {2i(n+1) \over \beta^2 }\left( \log \left( {x\over \alpha }\right)^{\beta }\right)^2 \cdot {c(i,n)^{\tau +1} \beta^{\tau +1}\over \alpha^{\tau +1}} \cdot { \left( \left( {x\over \alpha }\right)^{\beta } \right)^{{i\beta \tau + i\beta - \tau \over \beta}} x^{\beta -1} {1\over \alpha^{\beta -1}}\over \left( 1 + \left( {x\over \alpha }\right)^{\beta } \right)^{(n+1)(\tau + 1)+1}} dx \\
 = & {2i(n+1) \over \beta^2 } \cdot {c(i,n)^{\tau +1} \beta^{\tau +1}\over \alpha^{\tau +1}} {\alpha \over \beta } \cdot \int_0^{\infty } \left( \log t \right)^2 \cdot { t^{{i\beta \tau + i\beta - \tau \over \beta}} \over \left( 1 + t\right)^{(n+1)(\tau + 1)+1} } dt \\
= &  {2(n+1) \over \beta^2 } \cdot {c(i,n)^{\tau +1} \beta^{\tau }\over \alpha^{\tau }} \cdot B\left( {(n+1-i)\tau \beta + (n+1-i)\beta + \tau \over \beta }, {i\beta \tau + (i+1)\beta - \tau \over \beta }\right) \\
& \times \left[ \Psi' \left( {i\beta \tau + (i+1)\beta - \tau \over \beta }\right) +\Psi' \left( {(n+1-i)\tau \beta + (n+1-i)\beta + \tau \over \beta } \right)  \right. \\
&  \left. + \left( \Psi \left( {i\beta \tau + (i+1)\beta - \tau \over \beta }\right) - \Psi \left( {(n+1-i)\tau \beta + (n+1-i)\beta + \tau \over \beta } \right) \right)^2\right] .
\end{align*}

\subsection{Proof of Theorem 8}

We already know that

\begin{equation*}
\frac{\partial \log f_{\beta }^{(i)}(x)}{\partial \beta }= {1\over \beta } + i\log \left( {x\over \alpha }\right) - {(n+1) \over \left( 1 + \left( {x\over \alpha }\right)^{\beta }\right)} \left( {x\over \alpha }\right)^{\beta } \log \left( {x\over \alpha }\right).
\end{equation*}

Consequently,

\begin{align*}
\xi_{\tau }^{(i)}(\beta )  = & \int_0^{\infty } \left[ {1\over \beta } + i\log \left( {x\over \alpha }\right) - {(n+1) \over \left( 1 + \left( {x\over \alpha }\right)^{\beta }\right)} \left( {x\over \alpha }\right)^{\beta } \log \left( {x\over \alpha }\right) \right] \cdot {c(i,n)^{\tau +1} \beta^{\tau +1} x^{(i\beta -1)(\tau +1)} \alpha^{\beta (\tau +1)(n-i+1)} \over (\alpha^{\beta } + x^{\beta })^{(n+1)(\tau +1)}} dx \\
 = & D_1 + D_2 - D_3.
\end{align*}

Now,

\begin{eqnarray*}
D_1 & = & \int_{0}^{\infty } {1\over \beta } \cdot {c(i,n)^{\tau +1} \beta^{\tau +1} x^{(i\beta -1)(\tau +1)} \alpha^{\beta (\tau +1)(n-i+1)} \over (\alpha^{\beta } + x^{\beta })^{(n+1)(\tau +1)}} dx \\
& = & \beta C_1 \\
& = & \left( {\beta \over \alpha }\right)^{\tau } {c(i,n)^{\tau +1}\over \beta } B\left( {(n-i+1)\beta \tau + (n-i+1)\beta + \tau \over \beta }, {i\beta \tau + i\beta - \tau \over \beta }\right) .
\end{eqnarray*}

\begin{eqnarray*}
D_2 & = & \int_{0}^{\infty } i\log \left( {x\over \alpha }\right) \cdot {c(i,n)^{\tau +1} \beta^{\tau +1} x^{(i\beta -1)(\tau +1)} \alpha^{\beta (\tau +1)(n-i+1)} \over (\alpha^{\beta } + x^{\beta })^{(n+1)(\tau +1)}} dx \\
& = & {\beta \over 2} C_4 \\
& = & \left( {\beta \over \alpha }\right)^{\tau } {c(i,n)^{\tau +1} \cdot i \over \beta } B\left( {(n-i+1)\beta \tau + (n-i+1)\beta + \tau \over \beta }, {i\beta \tau + i\beta - \tau \over \beta }\right) \\
& & \left[ \Psi \left( {i\beta \tau + i\beta - \tau \over \beta } \right) - \Psi \left( {(n-i+1)\beta \tau + (n-i+1)\beta + \tau \over \beta }\right) \right] .
\end{eqnarray*}

\begin{eqnarray*}
D_3 & = & \int_{0}^{\infty } {(n+1) \over \left( 1 + \left( {x\over \alpha }\right)^{\beta }\right)} \left( {x\over \alpha }\right)^{\beta } \log \left( {x\over \alpha }\right) \cdot {c(i,n)^{\tau +1} \beta^{\tau +1} x^{(i\beta -1)(\tau +1)} \alpha^{\beta (\tau +1)(n-i+1)} \over (\alpha^{\beta } + x^{\beta })^{(n+1)(\tau +1)}} dx \\
& = & {\beta \over 2} C_5 \\
& = & \left( {\beta \over \alpha }\right)^{\tau } {c(i,n)^{\tau +1} \cdot (n+1) \over \beta } B\left( {(n-i+1)\beta \tau + (n-i+1)\beta + \tau \over \beta }, {i\beta \tau + (i+1)\beta - \tau \over \beta }\right) \\
& & \left[ \Psi \left( {i\beta \tau + (i+1)\beta - \tau \over \beta } \right) - \Psi \left( {(n-i+1)\beta \tau + (n-i+1)\beta + \tau \over \beta }\right) \right] \\
& = & \left( {\beta \over \alpha }\right)^{\tau } {c(i,n)^{\tau +1} \cdot (n+1) \over \beta } {i\beta \tau + i\beta - \tau \over (n+1)\beta \tau + (n+1)\beta } B\left( {(n-i+1)\beta \tau + (n-i+1)\beta + \tau \over \beta }, {i\beta \tau + i \beta - \tau \over \beta }\right) \\
& & \left[ \Psi \left( {i\beta \tau + i\beta - \tau \over \beta } \right) + {\beta \over i\beta \tau + i\beta - \tau }- \Psi \left( {(n-i+1)\beta \tau + (n-i+1)\beta + \tau \over \beta }\right) \right] \\
& = & \left( {\beta \over \alpha }\right)^{\tau } {c(i,n)^{\tau +1} \over \beta } {i\beta \tau + i\beta - \tau \over \beta (\tau + 1)} B\left( {(n-i+1)\beta \tau + (n-i+1)\beta + \tau \over \beta }, {i\beta \tau + i \beta - \tau \over \beta }\right) \\
& & \left[ \Psi \left( {i\beta \tau + i\beta - \tau \over \beta } \right) + {\beta \over i\beta \tau + i\beta - \tau }- \Psi \left( {(n-i+1)\beta \tau + (n-i+1)\beta + \tau \over \beta }\right) \right] .
\end{eqnarray*}

Finally,

\begin{align*}
D_1 + D_2 -D_3 = & \left( {\beta \over \alpha }\right)^{\tau } {c(i,n)^{\tau +1}\over \beta } B\left( {(n-i+1)\beta \tau + (n-i+1)\beta + \tau \over \beta }, {i\beta \tau + i\beta - \tau \over \beta }\right) \left[ 1- {1\over \tau +1} \right] +\\
 & + \left( {\beta \over \alpha }\right)^{\tau } {c(i,n)^{\tau +1} \over \beta } B\left( {(n-i+1)\beta \tau + (n-i+1)\beta + \tau \over \beta }, {i\beta \tau + i\beta - \tau \over \beta }\right) \\
 & \cdot \left[ i - {\beta \over i\beta \tau + i\beta - \tau } \right]  \cdot \left[ \Psi \left( {i\beta \tau + i\beta - \tau \over \beta } \right) - \Psi \left( {(n-i+1)\beta \tau + (n-i+1)\beta + \tau \over \beta }\right) \right] \\
 = & \left( {\beta \over \alpha }\right)^{\tau } {c(i,n)^{\tau +1}\over \beta } B\left( {(n-i+1)\beta \tau + (n-i+1)\beta + \tau \over \beta }, {i\beta \tau + i\beta - \tau \over \beta }\right) {\tau \over \tau +1} \\
 & \cdot \left[ 1+ {1\over \beta} \left( \Psi \left( {i\beta \tau + i\beta - \tau \over \beta } \right) - \Psi \left( {(n-i+1)\beta \tau + (n-i+1)\beta + \tau \over \beta }\right) \right)\right] .
\end{align*}

\subsection{Proof of Theorem 9}

By definition,

\begin{equation*}
\text{ }J_{\tau }^{(i)12}\left( \alpha ,\beta \right) =\dint\limits_{0}^{\infty
}\left( \frac{\partial \log f_{\alpha ,\beta }(x)^{(i)}}{\partial \alpha }\right)
\left( \frac{\partial \log f_{\alpha ,\beta }^{(i)}(x)}{\partial \beta }\right)
f_{\alpha , \beta }^{(i)}(x)^{\tau +1}dx.
\end{equation*}%

On the other hand, we already know that

\begin{equation*}
\frac{\partial \log f_{\alpha , \beta }^{(i)}(x)}{\partial \alpha }= \frac{\beta (n-i+1)}{\alpha }%
-\frac{(n+1)}{1 + \left( {x\over \alpha }\right)^{\beta }}\cdot \frac{\beta}{\alpha }
\text{ and }\frac{\partial \log f_{\alpha , \beta }^{(i)}(x)}{\partial \beta }=\frac{1}{\beta } + i \log \frac{x}{\alpha }-\frac{(n+1)}{1 + \left( {x\over \alpha }\right)^{\beta }} \left( {x\over \alpha }\right)^{\beta } \log \frac{x}{\alpha } .
\end{equation*}

Therefore, $\text{ }J_{\tau }^{(i)12}\left( \alpha ,\beta \right) $ is given by six terms corresponding to the product

\begin{eqnarray*}
\frac{\partial \log f_{\alpha , \beta }^{(i)}(x)}{\partial \alpha } \cdot \frac{\partial \log f_{\alpha , \beta }^{(i)}(x)}{\partial \beta } & = & {n-i+1\over \alpha } \\
& & +{\beta \over \alpha }i(n-i+1)\log \frac{x}{\alpha } \\
& & -{(n-i+1)(n+1)\beta \over \alpha }\cdot {\left( {x\over \alpha }\right)^{\beta } \log \frac{x}{\alpha } \over 1 + \left( {x\over \alpha }\right)^{\beta }} \\
& & - {(n+1)\over \alpha \left( 1 + \left( {x\over \alpha }\right)^{\beta }\right) } \\
& & - {(n+1) i \beta \over \alpha }\cdot  {\log \frac{x}{\alpha } \over 1 + \left( {x\over \alpha }\right)^{\beta }}\\
& & + {(n+1)^2 \beta \over \alpha }\cdot {\left( {x\over \alpha }\right)^{\beta } \log \frac{x}{\alpha } \over \left( 1 + \left( {x\over \alpha }\right)^{\beta }\right)^2} .
\end{eqnarray*}

Therefore,

$$ J_{\tau }^{(i)12}\left( \alpha ,\beta \right) = E_1 + E_2 -E_3 - E_4 - E_5 + E_6.$$

Let us compute each of them separately.

\begin{align*}
E_1 = & \int_0^{\infty } {n-i+1\over \alpha }\cdot \frac{c(i, n)^{\tau +1} \beta ^{\tau +1} \left( {x\over \alpha }\right)^{(i\beta -1)(\tau +1)}}{\alpha^{\tau +1}\left( 1 + \left( {x\over \alpha }\right)^{\beta }\right)^{(n+1)(\tau +1)}} dx \\
 = & {n-i+1\over \alpha } c(i, n)^{\tau +1} \left( {\beta \over \alpha }\right)^{\tau +1} \int_0^{\infty } \frac{\left( {x\over \alpha }\right)^{(i\beta -1)(\tau +1)}}{\left( 1 + \left( {x\over \alpha }\right)^{\beta }\right)^{(n+1)(\tau +1)}} dx \\
 = & {n-i+1\over \alpha } c(i, n)^{\tau +1} \left( {\beta \over \alpha }\right)^{\tau +1} \int_0^{\infty } \frac{\left( \left( {x\over \alpha }\right)^{\beta }\right)^{{(i\beta -1)(\tau +1)\over \beta }}}{\left( 1 + \left( {x\over \alpha }\right)^{\beta }\right)^{(n+1)(\tau +1)}} dx \\
 = & {n-i+1\over \alpha } c(i, n)^{\tau +1} \left( {\beta \over \alpha }\right)^{\tau +1} {\alpha \over \beta } \\
 & \times B\left( (n+1)(\tau +1) - {(i\beta -1)(\tau +1)\over \beta } - {1-\beta \over \beta } -1, {(i\beta -1)(\tau +1)\over \beta }+ {1-\beta \over \beta } + 1\right) \\
 = & {n-i+1\over \alpha } c(i, n)^{\tau +1} \left( {\beta \over \alpha }\right)^{\tau } B\left( {(n-i+1)\beta \tau + (n-i+1)\beta + \tau \over \beta }, {i\beta \tau + i\beta -\tau \over \beta }\right) ,
\end{align*}
where we are applying Corollary \ref{corI1} for the last equality.

\begin{eqnarray*}
E_2 & = & \int_0^{\infty } {\beta \over \alpha }i(n-i+1) \log \left( {x\over \alpha }\right) \cdot \frac{c(i, n)^{\tau +1} \beta ^{\tau +1} \left( {x\over \alpha }\right)^{(i\beta -1)(\tau +1)}}{\alpha^{\tau +1}\left( 1 + \left( {x\over \alpha }\right)^{\beta }\right)^{(n+1)(\tau +1)}} dx \\
& = & {\beta \over \alpha }i(n-i+1) c(i, n)^{\tau +1} \left( {\beta \over \alpha }\right)^{\tau +1} \int_0^{\infty } {1\over \beta }\log \left( {x\over \alpha }\right)^{\beta } \frac{\left( {x\over \alpha }\right)^{(i\beta -1)(\tau +1)}}{\left( 1 + \left( {x\over \alpha }\right)^{\beta }\right)^{(n+1)(\tau +1)}} dx \\
& = & {i(n-i+1)\over \alpha } c(i, n)^{\tau +1} \left( {\beta \over \alpha }\right)^{\tau +1} \int_0^{\infty } \log \left( {x\over \alpha }\right)^{\beta } \frac{\left( \left( {x\over \alpha }\right)^{\beta }\right)^{{(i\beta -1)(\tau +1)\over \beta }}}{\left( 1 + \left( {x\over \alpha }\right)^{\beta }\right)^{(n+1)(\tau +1)}} dx \\
& = & {i(n-i+1)\over \alpha } c(i, n)^{\tau +1} \left( {\beta \over \alpha }\right)^{\tau +1} {\alpha \over \beta } \\
& & B\left( (n+1)(\tau +1) - {(i\beta -1)(\tau +1)\over \beta } - {1-\beta \over \beta } -1, {(i\beta -1)(\tau +1)\over \beta }+ {1-\beta \over \beta } + 1\right) \\
& & \left[ \Psi \left( {(i\beta -1)(\tau +1)\over \beta }+ {1-\beta \over \beta } + 1\right) - \Psi \left( (n+1)(\tau +1) - {(i\beta -1)(\tau +1)\over \beta } - {1-\beta \over \beta } -1\right) \right] \\
& = & {n-i+1\over \alpha } c(i, n)^{\tau +1} \left( {\beta \over \alpha }\right)^{\tau } B\left( {(n-i+1)\beta \tau + (n-i+1)\beta + \tau \over \beta }, {i\beta \tau + i\beta -\tau \over \beta }\right) \\
& & \left[ \Psi \left( {i\beta \tau + i\beta -\tau \over \beta }\right) - \Psi \left( {(n-i+1)\beta \tau + (n-i+1)\beta + \tau \over \beta }\right) \right] ,
\end{eqnarray*}
where we are applying Corollary \ref{corI1} for the last equality.

\begin{eqnarray*}
E_3 & = & \int_0^{\infty } {\beta \over \alpha }(n+1)(n-i+1) \log \left( {x\over \alpha }\right) {\left( {x\over \alpha }\right)^{\beta }\over 1 + \left( {x\over \alpha }\right)^{\beta }} \cdot \frac{c(i, n)^{\tau +1} \beta ^{\tau +1} \left( {x\over \alpha }\right)^{(i\beta -1)(\tau +1)}}{\alpha^{\tau +1}\left( 1 + \left( {x\over \alpha }\right)^{\beta }\right)^{(n+1)(\tau +1)}} dx \\
& = & {\beta \over \alpha }(n+1)(n-i+1) c(i, n)^{\tau +1} \left( {\beta \over \alpha }\right)^{\tau +1} \int_0^{\infty } {1\over \beta }\log \left( {x\over \alpha }\right)^{\beta } {\left( {x\over \alpha }\right)^{\beta }\over 1 + \left( {x\over \alpha }\right)^{\beta }} \frac{\left( {x\over \alpha }\right)^{(i\beta -1)(\tau +1)}}{\left( 1 + \left( {x\over \alpha }\right)^{\beta }\right)^{(n+1)(\tau +1)}} dx \\
& = & {(n+1)(n-i+1)\over \alpha } c(i, n)^{\tau +1} \left( {\beta \over \alpha }\right)^{\tau +1} \int_0^{\infty } \log \left( {x\over \alpha }\right)^{\beta } \frac{\left( \left( {x\over \alpha }\right)^{\beta }\right)^{{(i\beta -1)(\tau +1)\over \beta }+1}}{\left( 1 + \left( {x\over \alpha }\right)^{\beta }\right)^{(n+1)(\tau +1)}+1} dx \\
& = & {(n+1)(n-i+1)\over \alpha } c(i, n)^{\tau +1} \left( {\beta \over \alpha }\right)^{\tau +1} {\alpha \over \beta } \\
& & B\left( (n+1)(\tau +1) - {(i\beta -1)(\tau +1)\over \beta } - {1-\beta \over \beta } -1, {(i\beta -1)(\tau +1)\over \beta }+ {1-\beta \over \beta } + 2\right) \\
& & \left[ \Psi \left( {(i\beta -1)(\tau +1)\over \beta }+ {1-\beta \over \beta } + 2\right) - \Psi \left( (n+1)(\tau +1) - {(i\beta -1)(\tau +1)\over \beta } - {1-\beta \over \beta } -1\right) \right] \\
& = & {(n+1)(n-i+1)\over \alpha } c(i, n)^{\tau +1} \left( {\beta \over \alpha }\right)^{\tau } B\left( {(n-i+1)\beta \tau + (n-i+1)\beta + \tau \over \beta }, {i\beta \tau + (i+1)\beta -\tau \over \beta }\right) \\
& & \left[ \Psi \left( {i\beta \tau + (i+1)\beta -\tau \over \beta }\right) - \Psi \left( {(n-i+1)\beta \tau + (n-i+1)\beta + \tau \over \beta }\right) \right] .
\end{eqnarray*}

\begin{eqnarray*}
E_4 & = & \int_0^{\infty } {n+1\over \alpha } {1 \over 1 + \left( {x\over \alpha }\right)^{\beta }} \cdot \frac{c(i, n)^{\tau +1} \beta ^{\tau +1} \left( {x\over \alpha }\right)^{(i\beta -1)(\tau +1)}}{\alpha^{\tau +1}\left( 1 + \left( {x\over \alpha }\right)^{\beta }\right)^{(n+1)(\tau +1)}} dx \\
& = & {n+1\over \alpha } c(i, n)^{\tau +1} \left( {\beta \over \alpha }\right)^{\tau +1} \int_0^{\infty } \frac{\left( {x\over \alpha }\right)^{(i\beta -1)(\tau +1)}}{\left( 1 + \left( {x\over \alpha }\right)^{\beta }\right)^{(n+1)(\tau +1)}+1} dx \\
& = & {n+1\over \alpha } c(i, n)^{\tau +1} \left( {\beta \over \alpha }\right)^{\tau +1} \int_0^{\infty } \frac{\left( \left( {x\over \alpha }\right)^{\beta }\right)^{{(i\beta -1)(\tau +1)\over \beta }}}{\left( 1 + \left( {x\over \alpha }\right)^{\beta }\right)^{(n+1)(\tau +1)}+1} dx \\
& = & {n+1\over \alpha } c(i, n)^{\tau +1} \left( {\beta \over \alpha }\right)^{\tau +1} {\alpha \over \beta }\\
&  & \times B\left( (n+1)(\tau +1) - {(i\beta -1)(\tau +1)\over \beta } - {1-\beta \over \beta }, {(i\beta -1)(\tau +1)\over \beta }+ {1-\beta \over \beta } + 1\right) \\
& = & {n+1\over \alpha } c(i, n)^{\tau +1} \left( {\beta \over \alpha }\right)^{\tau } B\left( {(n-i+1)\beta \tau + (n-i+2)\beta + \tau \over \beta }, {i\beta \tau + i\beta -\tau \over \beta }\right) .
\end{eqnarray*}

\begin{eqnarray*}
E_5 & = & \int_0^{\infty } {\beta \over \alpha }(n+1) i \log \left( {x\over \alpha }\right) {1 \over 1 + \left( {x\over \alpha }\right)^{\beta }} \cdot \frac{c(i, n)^{\tau +1} \beta ^{\tau +1} \left( {x\over \alpha }\right)^{(i\beta -1)(\tau +1)}}{\alpha^{\tau +1}\left( 1 + \left( {x\over \alpha }\right)^{\beta }\right)^{(n+1)(\tau +1)}} dx \\
& = & {\beta \over \alpha }(n+1) i c(i, n)^{\tau +1} \left( {\beta \over \alpha }\right)^{\tau +1} \int_0^{\infty } {1\over \beta }\log \left( {x\over \alpha }\right)^{\beta } {1\over 1 + \left( {x\over \alpha }\right)^{\beta }} \frac{\left( {x\over \alpha }\right)^{(i\beta -1)(\tau +1)}}{\left( 1 + \left( {x\over \alpha }\right)^{\beta }\right)^{(n+1)(\tau +1)}} dx \\
& = & {(n+1) i\over \alpha } c(i, n)^{\tau +1} \left( {\beta \over \alpha }\right)^{\tau +1} \int_0^{\infty } \log \left( {x\over \alpha }\right)^{\beta } \frac{\left( \left( {x\over \alpha }\right)^{\beta }\right)^{{(i\beta -1)(\tau +1)\over \beta }}}{\left( 1 + \left( {x\over \alpha }\right)^{\beta }\right)^{(n+1)(\tau +1)}+1} dx \\
& = & {(n+1) i\over \alpha } c(i, n)^{\tau +1} \left( {\beta \over \alpha }\right)^{\tau +1} {\alpha \over \beta } \\
& & B\left( (n+1)(\tau +1) - {(i\beta -1)(\tau +1)\over \beta } - {1-\beta \over \beta }, {(i\beta -1)(\tau +1)\over \beta }+ {1-\beta \over \beta } + 1\right) \\
& & \left[ \Psi \left( {(i\beta -1)(\tau +1)\over \beta }+ {1-\beta \over \beta } + 1\right) - \Psi \left( (n+1)(\tau +1) - {(i\beta -1)(\tau +1)\over \beta } - {1-\beta \over \beta }\right) \right] \\
& = & {(n+1) i \over \alpha } c(i, n)^{\tau +1} \left( {\beta \over \alpha }\right)^{\tau } B\left( {(n-i+1)\beta \tau + (n-i+2)\beta + \tau \over \beta }, {i\beta \tau + i\beta -\tau \over \beta }\right) \\
& & \left[ \Psi \left( {i\beta \tau + i\beta -\tau \over \beta }\right) - \Psi \left( {(n-i+1)\beta \tau + (n-i+2)\beta + \tau \over \beta }\right) \right] .
\end{eqnarray*}

\begin{eqnarray*}
E_6 & = & \int_0^{\infty } {\beta \over \alpha }(n+1)^2 \log \left( {x\over \alpha }\right) {\left( {x\over \alpha }\right)^{\beta }\over \left( 1 + \left( {x\over \alpha }\right)^{\beta }\right)^2} \cdot \frac{c(i, n)^{\tau +1} \beta ^{\tau +1} \left( {x\over \alpha }\right)^{(i\beta -1)(\tau +1)}}{\alpha^{\tau +1}\left( 1 + \left( {x\over \alpha }\right)^{\beta }\right)^{(n+1)(\tau +1)}} dx \\
& = & {\beta \over \alpha }(n+1)^2 c(i, n)^{\tau +1} \left( {\beta \over \alpha }\right)^{\tau +1} \int_0^{\infty } {1\over \beta }\log \left( {x\over \alpha }\right)^{\beta } {\left( {x\over \alpha }\right)^{\beta }\over \left( 1 + \left( {x\over \alpha }\right)^{\beta }\right) } \frac{\left( {x\over \alpha }\right)^{(i\beta -1)(\tau +1)}}{\left( 1 + \left( {x\over \alpha }\right)^{\beta }\right)^{(n+1)(\tau +1)}} dx \\
& = & {(n+1)^2\over \alpha } c(i, n)^{\tau +1} \left( {\beta \over \alpha }\right)^{\tau +1} \int_0^{\infty } \log \left( {x\over \alpha }\right)^{\beta } \frac{\left( \left( {x\over \alpha }\right)^{\beta }\right)^{{(i\beta -1)(\tau +1)\over \beta }+1}}{\left( 1 + \left( {x\over \alpha }\right)^{\beta }\right)^{(n+1)(\tau +1)}+2} dx \\
& = & {(n+1)^2\over \alpha } c(i, n)^{\tau +1} \left( {\beta \over \alpha }\right)^{\tau +1} {\alpha \over \beta } \\
& & B\left( (n+1)(\tau +1) - {(i\beta -1)(\tau +1)\over \beta } - {1-\beta \over \beta }, {(i\beta -1)(\tau +1)\over \beta }+ {1-\beta \over \beta } + 2\right) \\
& & \left[ \Psi \left( {(i\beta -1)(\tau +1)\over \beta }+ {1-\beta \over \beta } + 2\right) - \Psi \left( (n+1)(\tau +1) - {(i\beta -1)(\tau +1)\over \beta } - {1-\beta \over \beta }\right) \right] \\
& = & {(n+1)^2\over \alpha } c(i, n)^{\tau +1} \left( {\beta \over \alpha }\right)^{\tau } B\left( {(n-i+1)\beta \tau + (n-i+2)\beta + \tau \over \beta }, {i\beta \tau + (i+1)\beta -\tau \over \beta }\right) \\
& & \left[ \Psi \left( {i\beta \tau + (i+1)\beta -\tau \over \beta }\right) - \Psi \left( {(n-i+1)\beta \tau + (n-i+2)\beta + \tau \over \beta }\right) \right] .
\end{eqnarray*}

\subsection{Proof of Corollary 10}

It suffices to take $\tau =0$ in $E_1, ..., E_6.$ Hence,

\begin{eqnarray*}
E_1 & = & {n-i+1\over \alpha }c(i,n) B(n-i+1, i).\\
E_2 & = & {i(n-i+1)\over \alpha }c(i,n) B(n-i+1, i)\left[ \Psi (i) - \Psi (n-i+1)\right] .\\
E_3 & = & {(n-i+1)(n+1)\over \alpha }c(i,n) B(n-i+1, i+1)\left[ \Psi (i+1) - \Psi (n-i+1)\right] \\
& = & {(n-i+1)(n+1)\over \alpha }c(i,n) {i\over n+1} B(n-i+1, i)\left[ \Psi (i) +{1\over i} - \Psi (n-i+1)\right] \\
& = & {i (n-i+1)\over \alpha }c(i,n) B(n-i+1, i)\left[ \Psi (i) - \Psi (n-i+1)\right] + {(n-i+1)\over \alpha }c(i,n) B(n-i+1, i).
\end{eqnarray*}

Consequently, $E_1 + E_2 - E_3=0.$ Now,

\begin{eqnarray*}
E_4 & = & {n+1\over \alpha }c(i,n) B(n-i+2, i) \\
& = & {n+1\over \alpha }c(i,n) {n-i+1\over n+1} B(n-i+1, i) \\
& = & {n-i+1\over \alpha }c(i,n) B(n-i+1, i).
\end{eqnarray*}

\begin{eqnarray*}
E_5 & = & {i(n+1)\over \alpha }c(i,n) B(n-i+2, i)\left[ \Psi (i) - \Psi (n-i+2)\right] \\
& = & {i(n+1)\over \alpha }c(i,n) {n-i+1\over n+1} B(n-i+1, i)\left[ \Psi (i) - \Psi (n-i+1)- {1\over n-i+1}\right] \\
& = & {i (n-i+1)\over \alpha }c(i,n) B(n-i+1, i)\left[ \Psi (i) - \Psi (n-i+1)\right] - {i\over \alpha }c(i,n) B(n-i+1, i).
\end{eqnarray*}

\begin{eqnarray*}
E_6 & = & {(n+1)^2\over \alpha }c(i,n) B(n-i+2, i+1)\left[ \Psi (i+1) - \Psi (n-i+2)\right] \\
& = & {(n+1)^2\over \alpha }c(i,n) {n-i+1\over n+2} {i\over n+1} B(n-i+1, i)\left[ \Psi (i) - \Psi (n-i+1) {1\over i} -{1\over n-i+1}\right] \\
& = & {(n+1)i (n-i+1)\over (n+2) \alpha }c(i,n) B(n-i+1, i)\left[ \Psi (i) - \Psi (n-i+1)\right] \\
& & - {(n+1) (n-2i+1)\over (n+2) \alpha }c(i,n) B(n-i+1, i).
\end{eqnarray*}

Therefore,

\begin{eqnarray*}
E_6 - E_5 - E_4 & = & {c(i,n) B(n-i+1, i)\over \alpha }\left[ - {(n+1) (n-2i+1)\over (n+2)} +(n-i+1) -i \right] \\
& & +{c(i,n) B(n-i+1, i)\over \alpha }\left[ \Psi (i) - \Psi (n-i+1)\right] \left[ {(n-i+1) i (n+1)\over n+2} -i(n-i+1)\right] \\
& = & {c(i,n) B(n-i+1, i)\over \alpha }(n-2i+1)\left[ 1- {n+1\over n+2}\right] \\
& & + {c(i,n) B(n-i+1, i)\over \alpha }\left[ \Psi (i) - \Psi (n-i+1)\right] \left[ {n+1\over n+2} -1\right] \\
& = & {c(i,n) B(n-i+1, i)\over (n+2) \alpha } \left[ n-2i+1 - (n-i+1)i \left[ \Psi (i) - \Psi (n-i+1)\right] \right] .
\end{eqnarray*}

And taking into account that $c(i,n) B(n-i+1, i)=1,$ the result follows.

\newpage

\section{Complementary simulation results}

In the simulation section, we provided the results for different values of $n$ and ($\alpha = 1, \beta = 2.5$).
Here, we present the results for other combinations of $n$ and $\beta $ in order to show that the conclusions are the
same when varying these values.

In order to be self-contained, let us briefly outline the formulas for RM, SM and HL estimators.

Consider a sample $(x_1, ..., x_n)$ and assume the log-logistic model holds. Let us denote the ordered sample by $(x_{(1)}, ..., x_{(n)}).$ For the RM-estimator, we consider the values

\begin{equation*}
y_i= log \left[ {\frac{1}{1-F(x_{(i)})}} -1\right] ,\quad z_i=log (x_{(i)}),
\end{equation*}
where $F(x_{(i)})= {i\over n+1}.$ Next, we compute

\begin{equation*}
b_1= Med_{1\leq i\leq n} Med_{j\ne i} \left( {\frac{y_i- y_j}{z_i- z_j}}
\right) ,\quad b_0= Med_{1\leq i\leq n} (y_i - b_1 z_i),
\end{equation*}
and finally the RM-estimations for $\alpha $ and $\beta $ are given by

\begin{equation*}
\hat{\alpha }_{RM} = exp (-{\frac{b_0}{b_1}}),\quad \hat{\beta}_{RM}= b_1.
\end{equation*}

To compute the SM and HL estimators, we define $
z_i= log (x_i).
$

Next, we compute

\begin{equation*}
\hat{\mu }_{SM} = Med (z_1, ..., z_n),\quad \hat{\mu }_{HL} =
Med_{i<j}\left( {\frac{z_i + z_j}{2}}\right) ,
\end{equation*}

\begin{equation*}
\hat{s}_{SM} = {\frac{Med (|z_i - \hat{\mu }_{SM}|)}{\Phi^{-1}(3/4)}}, \quad
\hat{s}_{HL} = {\frac{Med_{i<j} (|z_i - z_j|)}{\sqrt{2} \Phi^{-1}(3/4)}},
\end{equation*}
with $\Phi^{-1}$ the inverse of the distribution function of a standard
normal. Then, the SM- and HL- estimations for $\alpha $ and $\beta $ are given, respectively, by

\begin{equation*}
\hat{\alpha }_{SM} = exp (\hat{\mu }_{SM}), \hat{\beta}_{SM}= {\frac{1}{\hat{s}_{SM}}},\quad \quad \hat{\alpha }_{HL} = exp (\hat{%
\mu }_{HL}), \hat{\beta}%
_{HL}= {\frac{1}{\hat{s}_{HL}}}.
\end{equation*}

\begin{table}
\begin{center}
\begin{tabular}{|c|cccc|}
\hline
& Bias & RMSE & $\hat{\alpha }$ & $\hat{\beta }$ \\ \hline
MLE         & 0.36673 & 0.36196 & 1.01193 & 1.57755 \\
$DPD_{0.1}$ & 0.36932 & 0.36600 & 1.00557 & 1.57785 \\
$DPD_{0.2}$ & 0.39371 & 0.39703 & 0.99852 & 1.59257 \\
$DPD_{0.3}$ & 0.43178 & 0.46292 & 0.99167 & 1.61555 \\
$DPD_{0.4}$ & 0.46650 & 0.50738 & 0.98510 & 1.63550 \\
$DPD_{0.5}$ & 0.49885 & 0.54768 & 0.97902 & 1.65519 \\
$DPD_{0.6}$ & 0.53514 & 0.60531 & 0.97303 & 1.68085 \\
$DPD_{0.7}$ & 0.57010 & 0.66090 & 0.96649 & 1.71013 \\
$DPD_{0.8}$ & 0.59823 & 0.69868 & 0.96066 & 1.73315 \\
$DPD_{0.9}$ & 0.63409 & 0.75694 & 0.95527 & 1.76218 \\
$DPD_{1.0}$ & 0.65697 & 0.79073 & 0.95020 & 1.78173 \\
RM          & 0.42076 & 0.40904 & 0.98584 & 1.41224 \\
SM          & 0.90421 & 0.77553 & 1.16817 & 0.87481 \\
HL          & 0.82004 & 0.72495 & 1.01568 & 0.81724 \\
\hline
\end{tabular}
\caption{Results for $n=10$ and $\beta =1.5$.}
\end{center}
\end{table}

\begin{table}
\begin{center}
\begin{tabular}{|c|cccc|}
\hline
& Bias & RMSE & $\hat{\alpha }$ & $\hat{\beta }$ \\ \hline
MLE         & 0.47564 & 0.55101 & 1.00422 & 2.62800 \\
$DPD_{0.1}$ & 0.47217 & 0.54805 & 1.00164 & 2.62059 \\
$DPD_{0.2}$ & 0.49444 & 0.57902 & 0.99891 & 2.63449 \\
$DPD_{0.3}$ & 0.53863 & 0.68286 & 0.99624 & 2.67013 \\
$DPD_{0.4}$ & 0.58501 & 0.75436 & 0.99370 & 2.70683 \\
$DPD_{0.5}$ & 0.64293 & 0.86396 & 0.99088 & 2.75772 \\
$DPD_{0.6}$ & 0.70101 & 0.96378 & 0.98817 & 2.80802 \\
$DPD_{0.7}$ & 0.76316 & 1.06942 & 0.98534 & 2.86487 \\
$DPD_{0.8}$ & 0.81998 & 1.15550 & 0.98266 & 2.91709 \\
$DPD_{0.9}$ & 0.89565 & 1.28145 & 0.97958 & 2.99121 \\
$DPD_{1.0}$ & 0.94722 & 1.35939 & 0.97694 & 3.04155 \\
RM          & 0.55745 & 0.62953 & 0.98809 & 2.35373 \\
SM          & 1.23750 & 1.18055 & 1.09027 & 1.45801 \\
HL          & 1.22112 & 1.17829 & 1.00602 & 1.36207 \\
\hline
\end{tabular}
\caption{Results for $n=10$ and $\beta =2.5$.}
\end{center}
\end{table}

\begin{table}
\begin{center}
\begin{tabular}{|c|cccc|}
\hline
& Bias & RMSE & $\hat{\alpha }$ & $\hat{\beta }$ \\ \hline
MLE         & 0.83677 & 1.08660 & 1.00100 & 5.25600 \\
$DPD_{0.1}$ & 0.82271 & 1.07093 & 1.00014 & 5.23326 \\
$DPD_{0.2}$ & 0.85020 & 1.11349 & 0.99927 & 5.25419 \\
$DPD_{0.3}$ & 0.93062 & 1.31988 & 0.99838 & 5.32875 \\
$DPD_{0.4}$ & 1.02072 & 1.50815 & 0.99763 & 5.41476 \\
$DPD_{0.5}$ & 1.11491 & 1.67100 & 0.99686 & 5.50457 \\
$DPD_{0.6}$ & 1.23986 & 1.92238 & 0.99598 & 5.62392 \\
$DPD_{0.7}$ & 1.36824 & 2.14666 & 0.99501 & 5.75036 \\
$DPD_{0.8}$ & 1.50255 & 2.39812 & 0.99412 & 5.88462 \\
$DPD_{0.9}$ & 1.61624 & 2.57792 & 0.99318 & 6.00189 \\
$DPD_{1.0}$ & 1.73744 & 2.75931 & 0.99211 & 6.12463 \\
RM          & 0.99281 & 1.24377 & 0.99275 & 4.70745 \\
SM          & 2.26485 & 2.33462 & 1.04152 & 2.91602 \\
HL          & 2.32061 & 2.34843 & 1.00175 & 2.72414 \\
\hline
\end{tabular}
\caption{Results for $n=10$ and $\beta =5.0$.}
\end{center}
\end{table}

\begin{table}
\begin{center}
\begin{tabular}{|c|cccc|}
\hline
& Bias & RMSE & $\hat{\alpha }$ & $\hat{\beta }$ \\ \hline
MLE         & 1.61649 & 2.17128 & 1.00023 & 10.51201 \\
$DPD_{0.1}$ & 1.58512 & 2.13466 & 0.99990 & 10.46085 \\
$DPD_{0.2}$ & 1.63064 & 2.21082 & 0.99959 & 10.49843 \\
$DPD_{0.3}$ & 1.78063 & 2.56874 & 0.99928 & 10.64058 \\
$DPD_{0.4}$ & 1.97119 & 3.01871 & 0.99903 & 10.82651 \\
$DPD_{0.5}$ & 2.16026 & 3.35303 & 0.99878 & 11.00950 \\
$DPD_{0.6}$ & 2.39309 & 3.83894 & 0.99854 & 11.23980 \\
$DPD_{0.7}$ & 2.65919 & 4.34485 & 0.99819 & 11.51039 \\
$DPD_{0.8}$ & 2.93641 & 4.86804 & 0.99787 & 11.79605 \\
$DPD_{0.9}$ & 3.14187 & 5.17449 & 0.99756 & 12.01123 \\
$DPD_{1.0}$ & 3.38205 & 5.54445 & 0.99725 & 12.25933 \\
RM          & 1.92418 & 2.48559 & 0.99605 & 9.41491 \\
SM          & 4.42956 & 4.66632 & 1.01991 & 5.83205 \\
HL          & 4.58067 & 4.69587 & 1.00057 & 5.44828 \\
\hline
\end{tabular}
\caption{Results for $n=10$ and $\beta =10.0$.}
\end{center}
\end{table}

\begin{table}
\begin{center}
\begin{tabular}{|c|cccc|}
\hline
& Bias & RMSE & $\hat{\alpha }$ & $\hat{\beta }$ \\ \hline
MLE         & 0.15503 & 0.14879 & 1.00309 & 1.51668 \\
$DPD_{0.1}$ & 0.15597 & 0.14853 & 1.00270 & 1.51438 \\
$DPD_{0.2}$ & 0.16643 & 0.15826 & 1.00229 & 1.51537 \\
$DPD_{0.3}$ & 0.17916 & 0.16970 & 1.00191 & 1.51711 \\
$DPD_{0.4}$ & 0.19037 & 0.18057 & 1.00150 & 1.51911 \\
$DPD_{0.5}$ & 0.20027 & 0.19043 & 1.00097 & 1.52147 \\
$DPD_{0.6}$ & 0.20888 & 0.19964 & 1.00026 & 1.52429 \\
$DPD_{0.7}$ & 0.21649 & 0.20775 & 0.99923 & 1.52769 \\
$DPD_{0.8}$ & 0.22320 & 0.21554 & 0.99811 & 1.53137 \\
$DPD_{0.9}$ & 0.23051 & 0.22361 & 0.99725 & 1.53449 \\
$DPD_{1.0}$ & 0.23746 & 0.23078 & 0.99632 & 1.53767 \\
RM          & 0.17040 & 0.15938 & 1.00428 & 1.44596 \\
SM          & 0.65352 & 0.59037 & 1.01232 & 0.94318 \\
HL          & 0.70996 & 0.66643 & 1.00635 & 0.84061 \\
\hline
\end{tabular}
\caption{Results for $n=25$ and $\beta =1.5$.}
\end{center}
\end{table}

\begin{table}
\begin{center}
\begin{tabular}{|c|cccc|}
\hline
& Bias & RMSE & $\hat{\alpha }$ & $\hat{\beta }$ \\ \hline
MLE         & 0.20709 & 0.22950 & 1.00142 & 2.52780 \\
$DPD_{0.1}$ & 0.20097 & 0.22035 & 1.00136 & 2.52370 \\
$DPD_{0.2}$ & 0.20563 & 0.22442 & 1.00121 & 2.52582 \\
$DPD_{0.3}$ & 0.21541 & 0.23429 & 1.00097 & 2.53087 \\
$DPD_{0.4}$ & 0.22655 & 0.24692 & 1.00065 & 2.53751 \\
$DPD_{0.5}$ & 0.23908 & 0.26103 & 1.00025 & 2.54501 \\
$DPD_{0.6}$ & 0.25187 & 0.27585 & 0.99980 & 2.55292 \\
$DPD_{0.7}$ & 0.26377 & 0.29081 & 0.99932 & 2.56094 \\
$DPD_{0.8}$ & 0.27500 & 0.30585 & 0.99884 & 2.56893 \\
$DPD_{0.9}$ & 0.28565 & 0.32077 & 0.99838 & 2.57670 \\
$DPD_{1.0}$ & 0.29595 & 0.33454 & 0.99794 & 2.58406 \\
RM          & 0.22880 & 0.24554 & 1.00206 & 2.40994 \\
SM          & 0.98601 & 0.96551 & 1.00566 & 1.57196 \\
HL          & 1.12924 & 1.10631 & 1.00333 & 1.40101 \\
\hline
\end{tabular}
\caption{Results for $n=25$ and $\beta =2.5$.}
\end{center}
\end{table}

\begin{table}
\begin{center}
\begin{tabular}{|c|cccc|}
\hline
& Bias & RMSE & $\hat{\alpha }$ & $\hat{\beta }$ \\ \hline
MLE         & 0.37099 & 0.45363 & 1.00055 & 5.05560 \\
$DPD_{0.1}$ & 0.35206 & 0.42820 & 1.00059 & 5.04766 \\
$DPD_{0.2}$ & 0.35341 & 0.42847 & 1.00058 & 5.05266 \\
$DPD_{0.3}$ & 0.36761 & 0.44552 & 1.00054 & 5.06431 \\
$DPD_{0.4}$ & 0.38836 & 0.47186 & 1.00046 & 5.07929 \\
$DPD_{0.5}$ & 0.41170 & 0.50278 & 1.00035 & 5.09585 \\
$DPD_{0.6}$ & 0.43482 & 0.53546 & 1.00023 & 5.11297 \\
$DPD_{0.7}$ & 0.45658 & 0.56795 & 1.00010 & 5.13002 \\
$DPD_{0.8}$ & 0.47792 & 0.59875 & 0.99996 & 5.14660 \\
$DPD_{0.9}$ & 0.49793 & 0.62739 & 0.99982 & 5.16269 \\
$DPD_{1.0}$ & 0.51671 & 0.65480 & 0.99968 & 5.17851 \\
RM          & 0.41113 & 0.48526 & 1.00084 & 4.81987 \\
SM          & 1.88554 & 1.92593 & 1.00218 & 3.14392 \\
HL          & 2.21307 & 2.21140 & 1.00149 & 2.80202 \\
\hline
\end{tabular}
\caption{Results for $n=25$ and $\beta =5.0$.}
\end{center}
\end{table}

\begin{table}
\begin{center}
\begin{tabular}{|c|cccc|}
\hline
& Bias & RMSE & $\hat{\alpha }$ & $\hat{\beta }$ \\ \hline
MLE         & 0.72041 & 0.90658 & 1.00023 & 10.11120 \\
$DPD_{0.1}$ & 0.67864 & 0.85190 & 1.00027 & 10.09599 \\
$DPD_{0.2}$ & 0.67814 & 0.84914 & 1.00029 & 10.10661 \\
$DPD_{0.3}$ & 0.70428 & 0.88296 & 1.00029 & 10.13060 \\
$DPD_{0.4}$ & 0.74547 & 0.93715 & 1.00027 & 10.16114 \\
$DPD_{0.5}$ & 0.79114 & 1.00115 & 1.00025 & 10.19468 \\
$DPD_{0.6}$ & 0.83663 & 1.06864 & 1.00021 & 10.22919 \\
$DPD_{0.7}$ & 0.88090 & 1.13530 & 1.00018 & 10.26340 \\
$DPD_{0.8}$ & 0.92174 & 1.19780 & 1.00014 & 10.29656 \\
$DPD_{0.9}$ & 0.96071 & 1.25461 & 1.00010 & 10.32857 \\
$DPD_{1.0}$ & 0.99714 & 1.30633 & 1.00006 & 10.35983 \\
RM          & 0.79906 & 0.96978 & 1.00037 & 9.63974 \\
SM          & 3.72798 & 3.85123 & 1.00093 & 6.28785 \\
HL          & 4.40349 & 4.42266 & 1.00070 & 5.60404 \\
\hline
\end{tabular}
\caption{Results for $n=25$ and $\beta =10.0$.}
\end{center}
\end{table}

\begin{table}
\begin{center}
\begin{tabular}{|c|cccc|}
\hline
& Bias & RMSE & $\hat{\alpha }$ & $\hat{\beta }$ \\ \hline
MLE         & 0.08380 &  0.07966 & 0.99982 & 1.50756 \\
$DPD_{0.1}$ & 0.08246 &  0.07838 & 0.99960 & 1.50630 \\
$DPD_{0.2}$ & 0.08781 &  0.08328 & 0.99947 & 1.50592 \\
$DPD_{0.3}$ & 0.09415 &  0.08875 & 0.99937 & 1.50570 \\
$DPD_{0.4}$ & 0.09978 &  0.09357 & 0.99922 & 1.50560 \\
$DPD_{0.5}$ & 0.10464 &  0.09763 & 0.99900 & 1.50565 \\
$DPD_{0.6}$ & 0.10877 &  0.10108 & 0.99871 & 1.50583 \\
$DPD_{0.7}$ & 0.11232 &  0.10409 & 0.99837 & 1.50613 \\
$DPD_{0.8}$ & 0.11543 &  0.10682 & 0.99798 & 1.50649 \\
$DPD_{0.9}$ & 0.11821 &  0.10936 & 0.99758 & 1.50692 \\
$DPD_{1.0}$ & 0.12074 &  0.11181 & 0.99717 & 1.50739 \\
RM          & 0.09302 & 0.08723 & 0.99808 & 1.45694 \\
SM          & 0.66464 & 0.61029 & 1.03188 & 0.90040 \\
HL          & 0.66722 & 0.64313 & 0.99968 & 0.85875 \\
\hline
\end{tabular}
\caption{Results for $n=50$ and $\beta =1.5$.}
\end{center}
\end{table}

\begin{table}
\begin{center}
\begin{tabular}{|c|cccc|}
\hline
& Bias & RMSE & $\hat{\alpha }$ & $\hat{\beta }$ \\ \hline
MLE         & 0.11258 & 0.12330 & 0.99977 & 2.51259 \\
$DPD_{0.1}$ & 0.10531 & 0.11463 & 0.99971 & 2.50999 \\
$DPD_{0.2}$ & 0.10608 & 0.11485 & 0.99968 & 2.50918 \\
$DPD_{0.3}$ & 0.11027 & 0.11828 & 0.99967 & 2.50902 \\
$DPD_{0.4}$ & 0.11515 & 0.12289 & 0.99964 & 2.50922 \\
$DPD_{0.5}$ & 0.12008 & 0.12790 & 0.99958 & 2.50969 \\
$DPD_{0.6}$ & 0.12512 & 0.13299 & 0.99950 & 2.51036 \\
$DPD_{0.7}$ & 0.13032 & 0.13804 & 0.99939 & 2.51119 \\
$DPD_{0.8}$ & 0.13532 & 0.14297 & 0.99927 & 2.51214 \\
$DPD_{0.9}$ & 0.14002 & 0.14776 & 0.99913 & 2.51319 \\
$DPD_{1.0}$ & 0.14444 & 0.15243 & 0.99899 & 2.51430 \\
RM          & 0.12599 & 0.13563 & 0.99871 & 2.42824 \\
SM          & 1.03790 & 1.00918 & 1.01836 & 1.50067 \\
HL          & 1.08433 & 1.07070 & 0.99968 & 1.43126 \\ \hline
\end{tabular}%
\end{center}
\caption{Results for $n=50$ and $\beta = 2.5.$}
\end{table}

\begin{table}
\begin{center}
\begin{tabular}{|c|cccc|}
\hline
& Bias & RMSE & $\hat{\alpha }$ & $\hat{\beta }$ \\ \hline
MLE         & 0.20229 & 0.24384 & 0.99984 & 5.02519 \\
$DPD_{0.1}$ & 0.18330 & 0.22067 & 0.99984 & 5.01953 \\
$DPD_{0.2}$ & 0.17952 & 0.21547 & 0.99985 & 5.01793 \\
$DPD_{0.3}$ & 0.18473 & 0.21992 & 0.99987 & 5.01799 \\
$DPD_{0.4}$ & 0.19297 & 0.22885 & 0.99990 & 5.01887 \\
$DPD_{0.5}$ & 0.20197 & 0.23957 & 0.99992 & 5.02030 \\
$DPD_{0.6}$ & 0.21131 & 0.25076 & 0.99994 & 5.02218 \\
$DPD_{0.7}$ & 0.22059 & 0.26183 & 0.99996 & 5.02444 \\
$DPD_{0.8}$ & 0.22926 & 0.27252 & 0.99997 & 5.02703 \\
$DPD_{0.9}$ & 0.23721 & 0.28274 & 0.99998 & 5.02992 \\
$DPD_{1.0}$ & 0.24474 & 0.29248 & 0.99999 & 5.03304 \\
RM          & 0.22744 & 0.26840 & 0.99930 & 4.85648 \\
SM          & 2.01779 & 2.01620 & 1.00890 & 3.00133 \\
HL          & 2.14528 & 2.14106 & 0.99979 & 2.86251 \\
\hline
\end{tabular}%
\end{center}
\caption{Results for $n=50$ and $\beta = 5.0.$}
\end{table}

\begin{table}
\begin{center}
\begin{tabular}{|c|cccc|}
\hline
& Bias & RMSE & $\hat{\alpha }$ & $\hat{\beta }$ \\ \hline
MLE         & 0.39314 & 0.48733 & 0.99991 & 10.05038 \\
$DPD_{0.1}$ & 0.35297 & 0.43752 & 0.99991 & 10.03879 \\
$DPD_{0.2}$ & 0.34244 & 0.42443 & 0.99993 & 10.03574 \\
$DPD_{0.3}$ & 0.35139 & 0.43261 & 0.99995 & 10.03634 \\
$DPD_{0.4}$ & 0.36737 & 0.45089 & 0.99997 & 10.03874 \\
$DPD_{0.5}$ & 0.38509 & 0.47316 & 1.00000 & 10.04232 \\
$DPD_{0.6}$ & 0.40285 & 0.49647 & 1.00002 & 10.04683 \\
$DPD_{0.7}$ & 0.42052 & 0.51945 & 1.00005 & 10.05215 \\
$DPD_{0.8}$ & 0.43690 & 0.54155 & 1.00008 & 10.05816 \\
$DPD_{0.9}$ & 0.45232 & 0.56259 & 1.00010 & 10.06476 \\
$DPD_{1.0}$ & 0.46632 & 0.58254 & 1.00013 & 10.07184 \\
RM         & 0.44259 & 0.53645 & 0.99964 & 9.71296 \\
SM         & 4.00686 & 4.03215 & 1.00438 & 6.00266 \\
HL         & 4.27887 & 4.28209 & 0.99989 & 5.72503 \\
\hline
\end{tabular}%
\end{center}
\caption{Results for $n=50$ and $\beta = 10.0.$}
\end{table}

\begin{table}
\begin{center}
\begin{tabular}{|c|cccc|}
\hline
& Bias & RMSE & $\hat{\alpha }$ & $\hat{\beta }$ \\ \hline
MLE         & 0.05552 & 0.05357 & 1.00047 & 1.50174 \\
$DPD_{0.1}$ & 0.05494 & 0.05199 & 1.00033 & 1.50154 \\
$DPD_{0.2}$ & 0.05884 & 0.05510 & 1.00019 & 1.50175 \\
$DPD_{0.3}$ & 0.06308 & 0.05863 & 0.99999 & 1.50209 \\
$DPD_{0.4}$ & 0.06696 & 0.06188 & 0.99974 & 1.50255 \\
$DPD_{0.5}$ & 0.07034 & 0.06479 & 0.99943 & 1.50310 \\
$DPD_{0.6}$ & 0.07330 & 0.06742 & 0.99909 & 1.50370 \\
$DPD_{0.7}$ & 0.07591 & 0.06983 & 0.99874 & 1.50431 \\
$DPD_{0.8}$ & 0.07830 & 0.07209 & 0.99840 & 1.50491 \\
$DPD_{0.9}$ & 0.08055 & 0.07424 & 0.99808 & 1.50548 \\
$DPD_{1.0}$ & 0.08272 & 0.07631 & 0.99777 & 1.50602 \\
RM         & 0.06222 & 0.05875 & 1.00098 & 1.47066 \\
SM         & 0.62188 & 0.58480 & 1.00168 & 0.92068 \\
HL         & 0.65325 & 0.63691 & 1.00032 & 0.86398 \\
\hline
\end{tabular}%
\end{center}
\caption{Results for $n=75$ and $\beta = 1.5.$}
\end{table}

\begin{table}
\begin{center}
\begin{tabular}{|c|cccc|}
\hline
& Bias & RMSE & $\hat{\alpha }$ & $\hat{\beta }$ \\ \hline
MLE         & 0.07467 & 0.08292 & 1.00023 & 2.50289 \\
$DPD_{0.1}$ & 0.07018 & 0.07601 & 1.00016 & 2.50253 \\
$DPD_{0.2}$ & 0.07120 & 0.07647 & 1.00008 & 2.50294 \\
$DPD_{0.3}$ & 0.07442 & 0.07938 & 0.99998 & 2.50356 \\
$DPD_{0.4}$ & 0.07821 & 0.08315 & 0.99984 & 2.50427 \\
$DPD_{0.5}$ & 0.08223 & 0.08717 & 0.99969 & 2.50502 \\
$DPD_{0.6}$ & 0.08619 & 0.09117 & 0.99952 & 2.50576 \\
$DPD_{0.7}$ & 0.08992 & 0.09498 & 0.99935 & 2.50647 \\
$DPD_{0.8}$ & 0.09330 & 0.09856 & 0.99918 & 2.50714 \\
$DPD_{0.9}$ & 0.09643 & 0.10190 & 0.99902 & 2.50776 \\
$DPD_{1.0}$ & 0.09939 & 0.10503 & 0.99886 & 2.50835 \\
RM          & 0.08437 & 0.09135 & 1.00053 & 2.45111 \\
SM          & 0.99106 & 0.97106 & 1.00066 & 1.53447 \\
HL          & 1.07037 & 1.06097 & 1.00014 & 1.43996 \\
\hline
\end{tabular}%
\end{center}
\caption{Results for $n=75$ and $\beta = 2.5.$}
\end{table}

\begin{table}
\begin{center}
\begin{tabular}{|c|cccc|}
\hline
& Bias & RMSE & $\hat{\alpha }$ & $\hat{\beta }$ \\ \hline
MLE         & 0.13427 & 0.16399 & 1.00009 & 5.00579 \\
$DPD_{0.1}$ & 0.12247 & 0.14714 & 1.00007 & 5.00505 \\
$DPD_{0.2}$ & 0.12062 & 0.14523 & 1.00004 & 5.00578 \\
$DPD_{0.3}$ & 0.12450 & 0.15010 & 1.00000 & 5.00678 \\
$DPD_{0.4}$ & 0.13070 & 0.15744 & 0.99995 & 5.00775 \\
$DPD_{0.5}$ & 0.13752 & 0.16533 & 0.99990 & 5.00870 \\
$DPD_{0.6}$ & 0.14403 & 0.17301 & 0.99984 & 5.00964 \\
$DPD_{0.7}$ & 0.15041 & 0.18019 & 0.99979 & 5.01063 \\
$DPD_{0.8}$ & 0.15624 & 0.18679 & 0.99973 & 5.01168 \\
$DPD_{0.9}$ & 0.16144 & 0.19284 & 0.99967 & 5.01281 \\
$DPD_{1.0}$ & 0.16618 & 0.19839 & 0.99961 & 5.01401 \\
RM         & 0.15243 & 0.18079 & 1.00024 & 4.90221 \\
SM         & 1.94382 & 1.94111 & 1.00020 & 3.06894 \\
HL         & 2.12524 & 2.12178 & 1.00005 & 2.87993 \\
\hline
\end{tabular}%
\end{center}
\caption{Results for $n=75$ and $\beta = 5.0.$}
\end{table}

\begin{table}
\begin{center}
\begin{tabular}{|c|cccc|}
\hline
& Bias & RMSE & $\hat{\alpha }$ & $\hat{\beta }$ \\ \hline
MLE         & 0.26099 & 0.32774 & 1.00004 & 10.01158 \\
$DPD_{0.1}$ & 0.23603 & 0.29278 & 1.00003 & 10.01009 \\
$DPD_{0.2}$ & 0.23087 & 0.28783 & 1.00002 & 10.01140 \\
$DPD_{0.3}$ & 0.23739 & 0.29724 & 1.00000 & 10.01310 \\
$DPD_{0.4}$ & 0.24880 & 0.31180 & 0.99999 & 10.01465 \\
$DPD_{0.5}$ & 0.26117 & 0.32741 & 0.99997 & 10.01611 \\
$DPD_{0.6}$ & 0.27319 & 0.34248 & 0.99995 & 10.01765 \\
$DPD_{0.7}$ & 0.28497 & 0.35652 & 0.99993 & 10.01938 \\
$DPD_{0.8}$ & 0.29591 & 0.36942 & 0.99991 & 10.02136 \\
$DPD_{0.9}$ & 0.30589 & 0.38128 & 0.99989 & 10.02362 \\
$DPD_{1.0}$ & 0.31468 & 0.39219 & 0.99987 & 10.02613 \\
RM         & 0.29671 & 0.36134 & 1.00011 & 9.80442 \\
SM         & 3.86850 & 3.88210 & 1.00007 & 6.13788 \\
HL         & 4.24273 & 4.24355 & 1.00002 & 5.75986 \\
\hline
\end{tabular}%
\end{center}
\caption{Results for $n=75$ and $\beta = 10.0.$}
\end{table}

\begin{table}
\begin{center}
\begin{tabular}{|c|cccc|}
\hline
& Bias & RMSE & $\hat{\alpha }$ & $\hat{\beta }$ \\ \hline
MLE         & 0.04296 & 0.04082 & 1.00005 & 1.50226 \\
$DPD_{0.1}$ & 0.04247 & 0.03971 & 1.00005 & 1.50162 \\
$DPD_{0.2}$ & 0.04547 & 0.04232 & 1.00001 & 1.50158 \\
$DPD_{0.3}$ & 0.04869 & 0.04513 & 0.99995 & 1.50172 \\
$DPD_{0.4}$ & 0.05143 & 0.04761 & 0.99988 & 1.50194 \\
$DPD_{0.5}$ & 0.05379 & 0.04975 & 0.99979 & 1.50220 \\
$DPD_{0.6}$ & 0.05581 & 0.05162 & 0.99970 & 1.50249 \\
$DPD_{0.7}$ & 0.05761 & 0.05329 & 0.99961 & 1.50277 \\
$DPD_{0.8}$ & 0.05930 & 0.05483 & 0.99952 & 1.50306 \\
$DPD_{0.9}$ & 0.06090 & 0.05630 & 0.99943 & 1.50335 \\
$DPD_{1.0}$ & 0.06255 & 0.05772 & 0.99934 & 1.50363 \\
RM         & 0.04815 & 0.04561 & 0.99919 & 1.47629 \\
SM         & 0.62856 & 0.59454 & 1.01370 & 0.90907 \\
HL         & 0.64668 & 0.63379 & 1.00037 & 0.86672 \\
\hline
\end{tabular}%
\end{center}
\caption{Results for $n=100$ and $\beta = 1.5.$}
\end{table}

\begin{table}
\begin{center}
\begin{tabular}{|c|cccc|}
\hline
& Bias & RMSE & $\hat{\alpha }$ & $\hat{\beta }$ \\ \hline
MLE         & 0.05753 & 0.06297 & 1.00000 & 2.50377 \\
$DPD_{0.1}$ & 0.05365 & 0.05741 & 1.00001 & 2.50267 \\
$DPD_{0.2}$ & 0.05452 & 0.05789 & 0.99999 & 2.50264 \\
$DPD_{0.3}$ & 0.05699 & 0.06026 & 0.99996 & 2.50304 \\
$DPD_{0.4}$ & 0.05998 & 0.06325 & 0.99991 & 2.50364 \\
$DPD_{0.5}$ & 0.06297 & 0.06638 & 0.99985 & 2.50434 \\
$DPD_{0.6}$ & 0.06590 & 0.06947 & 0.99978 & 2.50505 \\
$DPD_{0.7}$ & 0.06871 & 0.07240 & 0.99972 & 2.50574 \\
$DPD_{0.8}$ & 0.07138 & 0.07516 & 0.99967 & 2.50639 \\
$DPD_{0.9}$ & 0.07380 & 0.07774 & 0.99962 & 2.50700 \\
$DPD_{1.0}$ & 0.07608 & 0.08016 & 0.99957 & 2.50757 \\
RM          & 0.06522 & 0.07084 & 0.99948 & 2.46048 \\
SM          & 1.00733 & 0.98813 & 1.00796 & 1.51512 \\
HL          & 1.06350 & 1.05599 & 1.00019 & 1.44454 \\ \hline
\end{tabular}%
\end{center}
\caption{Results for $n=100$ and $\beta = 2.5.$}
\end{table}

\begin{table}
\begin{center}
\begin{tabular}{|c|cccc|}
\hline
& Bias & RMSE & $\hat{\alpha }$ & $\hat{\beta }$ \\ \hline
MLE         & 0.10319 & 0.12446 & 0.99999 & 5.00754 \\
$DPD_{0.1}$ & 0.09279 & 0.11053 & 1.00000 & 5.00540 \\
$DPD_{0.2}$ & 0.09203 & 0.10901 & 1.00000 & 5.00537 \\
$DPD_{0.3}$ & 0.09566 & 0.11314 & 0.99998 & 5.00622 \\
$DPD_{0.4}$ & 0.10117 & 0.11936 & 0.99996 & 5.00744 \\
$DPD_{0.5}$ & 0.10703 & 0.12607 & 0.99994 & 5.00881 \\
$DPD_{0.6}$ & 0.11268 & 0.13261 & 0.99992 & 5.01023 \\
$DPD_{0.7}$ & 0.11787 & 0.13873 & 0.99989 & 5.01164 \\
$DPD_{0.8}$ & 0.12253 & 0.14438 & 0.99987 & 5.01301 \\
$DPD_{0.9}$ & 0.12683 & 0.14957 & 0.99985 & 5.01435 \\
$DPD_{1.0}$ & 0.13092 & 0.15435 & 0.99983 & 5.01564 \\
RM         & 0.11775 & 0.14018 & 0.99973 & 4.92095 \\
SM         & 1.98093 & 1.97550 & 1.00388 & 3.03025 \\
HL         & 2.11494 & 2.11189 & 1.00008 & 2.88908 \\
\hline
\end{tabular}%
\end{center}
\caption{Results for $n=100$ and $\beta = 5.0.$}
\end{table}

\begin{table}
\begin{center}
\begin{tabular}{|c|cccc|}
\hline
& Bias & RMSE & $\hat{\alpha }$ & $\hat{\beta }$ \\ \hline
MLE         & 0.20044 & 0.24873 & 0.99999 & 10.01508 \\
$DPD_{0.1}$ & 0.17818 & 0.21954 & 1.00000 & 10.01092 \\
$DPD_{0.2}$ & 0.17579 & 0.21540 & 1.00000 & 10.01089 \\
$DPD_{0.3}$ & 0.18276 & 0.22352 & 0.99999 & 10.01253 \\
$DPD_{0.4}$ & 0.19328 & 0.23615 & 0.99999 & 10.01484 \\
$DPD_{0.5}$ & 0.20496 & 0.24974 & 0.99998 & 10.01743 \\
$DPD_{0.6}$ & 0.21578 & 0.26292 & 0.99997 & 10.02014 \\
$DPD_{0.7}$ & 0.22551 & 0.27523 & 0.99996 & 10.02288 \\
$DPD_{0.8}$ & 0.23456 & 0.28659 & 0.99995 & 10.02562 \\
$DPD_{0.9}$ & 0.24297 & 0.29703 & 0.99994 & 10.02833 \\
$DPD_{1.0}$ & 0.25078 & 0.30663 & 0.99993 & 10.03101 \\
RM         & 0.22916 & 0.28018 & 0.99986 & 9.84190 \\
SM         & 3.94509 & 3.95091 & 1.00192 & 6.06049 \\
HL         & 4.22385 & 4.22377 & 1.00004 & 5.77816 \\
\hline
\end{tabular}%
\end{center}
\caption{Results for $n=100$ and $\beta = 10.0.$}
\end{table}

In Tables 21 to 24, we include the data that have been considered to build Figures 1 to 4 in the paper.

\begin{sidewaystable}
\begin{center}
\resizebox{\textwidth}{!}{\begin{tabular}{cccccccccccccc}
\hline
MLE     & $DPD_{0.1}$ & $DPD_{0.2}$ & $DPD_{0.3}$ & $DPD_{0.4}$ & $DPD_{0.5}$ & $DPD_{0.6}$ & $DPD_{0.7}$ & $DPD_{0.8}$ & RM      & SM      & HL \\ \hline
0.12632 & 0.11165     & 0.11087     & 0.11520     & 0.12112     &  0.12743    &  0.13368    &  0.13971    &  0.14541    & 0.13975 & 1.96792 & 2.11103 \\
2.79244 & 0.50392     & 0.24375     & 0.20220     & 0.19519     &  0.19824    &  0.20519    &  0.21360    &  0.22230    & 0.51994 & 2.15191 & 2.43101 \\
3.65647 & 1.14888     & 0.45768     & 0.34101     & 0.31673     &  0.31965    &  0.33290    &  0.35003    &  0.36806    & 0.87505 & 2.32271 & 2.73765 \\
3.96925 & 2.35926     & 0.72293     & 0.51264     & 0.47168     &  0.47765    &  0.50009    &  0.52800    &  0.55661    & 1.23249 & 2.50961 & 3.04014 \\
3.99991 & 3.88676     & 1.08157     & 0.72456     & 0.65355     &  0.65562    &  0.68243    &  0.71722    &  0.75308    & 1.58023 & 2.69591 & 3.33873 \\
4.00007 & 3.99984     & 2.03909     & 0.99514     & 0.87329     &  0.86499    &  0.89334    &  0.93355    &  0.97579    & 1.94470 & 2.88669 & 3.66102 \\
4.00008 & 4.00038     & 3.72815     & 1.50177     & 1.15606     &  1.12048    &  1.14550    &  1.19343    &  1.23397    & 2.32659 & 3.09083 & 4.01093 \\
4.00011 & 4.00132     & 3.99733     & 2.72044     & 1.66613     &  1.45473    &  1.43524    &  1.47776    &  1.50158    & 2.71268 & 3.29234 & 4.38769 \\
4.00012 & 4.00301     & 4.00209     & 3.82119     & 2.83881     &  2.24735    &  2.13889    &  2.04122    &  1.97892    & 3.13622 & 3.51284 &  4.60939 \\
\hline
\end{tabular}}%
\end{center}
\par
\label{tablasimconcont2}
\caption{Results for Case 1.}
\end{sidewaystable}

\begin{sidewaystable}
\begin{center}
\resizebox{\textwidth}{!}{\begin{tabular}{cccccccccccccc}
\hline
MLE     & $DPD_{0.1}$ & $DPD_{0.2}$ & $DPD_{0.3}$ & $DPD_{0.4}$ & $DPD_{0.5}$ & $DPD_{0.6}$ & $DPD_{0.7}$ & $DPD_{0.8}$ & RM      & SM      & HL \\ \hline
0.12355 &  0.10807    & 0.10489     & 0.10745     & 0.11232     & 0.11799     & 0.12383     & 0.12957     & 0.13504     & 0.14346 & 1.97776 & 2.11525 \\
0.87226 &  0.64923    & 0.46066     & 0.33165     & 0.25890     & 0.22482     & 0.21263     & 0.21167     & 0.21617     & 0.47403 & 2.15495 & 2.43551 \\
1.49037 &  1.23164    & 0.94426     & 0.69598     & 0.52993     & 0.44030     & 0.40118     & 0.39050     & 0.39455     & 0.77591 & 2.38458 & 2.72341 \\
1.94563 &  1.73798    & 1.44769     & 1.12633     & 0.86459     & 0.70480     & 0.62710     & 0.59901     & 0.59744     & 1.08563 & 2.61851 & 2.98071 \\
2.27474 &  2.14570    & 1.92592     & 1.61591     & 1.28657     & 1.04358     & 0.91020     & 0.85262     & 0.83744     & 1.40805 & 2.87023 & 3.21243 \\
2.51379 &  2.45419    & 2.32259     & 2.09676     & 1.77666     & 1.45892     & 1.25011     & 1.14688     & 1.10795     & 1.80102 & 3.13408 & 3.41433 \\
2.69080 &  2.68653    & 2.63214     & 2.51479     & 2.30993     & 2.01350     & 1.72616     & 1.54541     & 1.45785     & 2.19305 & 3.42132 & 3.60479 \\
2.82244 &  2.85790    & 2.85641     & 2.81587     & 2.72694     & 2.56959     & 2.32908     & 2.07492     & 1.90390     & 2.50179 & 3.71815 & 3.75821 \\
2.91965 &  2.98466    & 3.01900     & 3.02543     & 3.00404     & 2.95139     & 2.85869     & 2.71289     & 2.52452     & 2.72515 & 3.99548 & 3.86485 \\
\hline
\end{tabular}}%
\end{center}
\par
\label{tablasimconcont2}
\caption{Results for Case 2.}
\end{sidewaystable}

\begin{sidewaystable}
\begin{center}
\resizebox{\textwidth}{!}{\begin{tabular}{cccccccccccccc}
\hline
MLE     & $DPD_{0.1}$ & $DPD_{0.2}$ & $DPD_{0.3}$ & $DPD_{0.4}$ & $DPD_{0.5}$ & $DPD_{0.6}$ & $DPD_{0.7}$ & $DPD_{0.8}$ & RM      & SM      & HL \\ \hline
0.12589 &  0.11152    & 0.10995     & 0.11358     & 0.11898     & 0.12483     & 0.13060     & 0.13608     & 0.14122     & 0.14328 & 1.97600 & 2.11241 \\
1.27777 &  0.64442    & 0.32032     & 0.21528     & 0.18952     & 0.18812     & 0.19467     & 0.20386     & 0.21364     & 0.46522 & 2.16102 & 2.42516 \\
2.06203 &  1.27998    & 0.65611     & 0.40214     & 0.32782     & 0.31584     & 0.32542     & 0.34237     & 0.36122     & 0.76126 & 2.35214 & 2.71561 \\
2.59347 &  1.89381    & 1.06596     & 0.63796     & 0.50322     & 0.47689     & 0.48833     & 0.51269     & 0.54044     & 1.05625 & 2.56784 & 2.99396 \\
2.95930 &  2.42607    & 1.55034     & 0.92734     & 0.71159     & 0.66236     & 0.67144     & 0.70060     & 0.73535     & 1.35907 & 2.79512 & 3.25626 \\
3.22429 &  2.85505    & 2.09778     & 1.29313     & 0.96755     & 0.88465     & 0.88701     & 0.91871     & 0.95897     & 1.71144 & 3.04216 & 3.50446 \\
3.42083 &  3.17806    & 2.64003     & 1.75981     & 1.27808     & 1.14042     & 1.12654     & 1.15528     & 1.19724     & 2.10080 & 3.29826 & 3.74301 \\
3.58067 &  3.42320    & 3.08301     & 2.33847     & 1.66653     & 1.44151     & 1.39837     & 1.41769     & 1.45735     & 2.54256 & 3.55278 & 3.99143 \\
3.72012 &  3.62031    & 3.41200     & 2.95737     & 2.19725     & 1.82155     & 1.72323     & 1.72113     & 1.75162     & 3.07281 & 3.80248 & 4.25151 \\
\hline
\end{tabular}}%
\end{center}
\par
\label{tablasimconcont2}
\caption{Results for Case 3.}
\end{sidewaystable}

\begin{sidewaystable}
\begin{center}
\resizebox{\textwidth}{!}{\begin{tabular}{cccccccccccccc}
\hline
MLE     & $DPD_{0.1}$ & $DPD_{0.2}$ & $DPD_{0.3}$ & $DPD_{0.4}$ & $DPD_{0.5}$ & $DPD_{0.6}$ & $DPD_{0.7}$ & $DPD_{0.8}$ & RM      & SM      & HL \\ \hline
0.12703 &  0.11096    & 0.10803     & 0.11061     & 0.11526     & 0.12055     & 0.12592     & 0.13112     & 0.13609     & 0.14086 & 1.97898 & 2.11389 \\
2.04462 &  0.53361    & 0.15189     & 0.13275     & 0.14324     & 0.15783     & 0.17263     & 0.18659     & 0.19934     & 0.46801 & 2.16097 & 2.43498 \\
2.97688 &  1.21611    & 0.25912     & 0.18733     & 0.21085     & 0.24633     & 0.28181     & 0.31464     & 0.34411     & 0.76764 & 2.37859 & 2.72843 \\
3.49092 &  2.05810    & 0.40965     & 0.26923     & 0.30785     & 0.36634     & 0.42335     & 0.47498     & 0.52054     & 1.08621 & 2.61033 & 2.99299 \\
3.81426 &  2.90145    & 0.60903     & 0.36654     & 0.41701     & 0.49778     & 0.57601     & 0.64609     & 0.70726     & 1.42199 & 2.86607 & 3.23862 \\
4.05861 &  3.51932    & 0.89918     & 0.49440     & 0.55437     & 0.65817     & 0.75809     & 0.84651     & 0.92272     & 1.82150 & 3.14369 & 3.47779 \\
4.21761 &  3.93500    & 1.34703     & 0.63635     & 0.69930     & 0.82514     & 0.94646     & 1.05279     & 1.14338     & 2.27052 & 3.42705 & 5.48892 \\
4.43936 &  4.24694    & 2.52282     & 0.81571     & 0.87183     & 1.01942     & 1.16261     & 1.28696     & 1.39157     & 2.85635 & 3.72076 & 6.38169 \\
4.47214 &  4.45733    & 4.15105     & 1.03388     & 1.06507     & 1.23320     & 1.39845     & 1.54079     & 1.65894     & 3.74137 & 4.02277 & 6.83677 \\
\hline
\end{tabular}}%
\end{center}
\par
\label{tablasimconcont2}
\caption{Results for Case 4.}
\end{sidewaystable}

\end{document}